\documentclass[11pt, oneside]{article}   	
\usepackage{geometry}                		
\usepackage[dvipdfm]{graphicx}		
								
\usepackage{authblk}

\usepackage{amsthm,amsmath,amssymb,graphicx,mathrsfs,braket,latexsym}
\newtheorem{thm}{Theorem}[section]
\newtheorem{prop}[thm]{Proposition}
\newtheorem{lem}[thm]{Lemma}
\newtheorem{cor}[thm]{Corollary}
\theoremstyle{definition}
\newtheorem{defi}[thm]{Definition}
\newtheorem{rem}[thm]{Remark}
\newtheorem{ex}[thm]{Example}

\newtheorem*{setting}{Setting}

\DeclareMathOperator{\supp}{supp}
\DeclareMathOperator{\card}{card}
\DeclareMathOperator{\diam}{diam}
\DeclareMathOperator{\Ker}{Ker}

\DeclareMathOperator{\spec}{sp}
\DeclareMathOperator{\spa}{span}

\newcommand{\Od}{\mbox{O}(d)}

\newcommand{\Zpo}{\mathbb{Z}_{>0}}
\newcommand{\Znn}{\mathbb{Z}_{\geq0}}
\newcommand{\calT}{\mathcal{T}}
\newcommand{\calA}{\mathcal{A}}
\newcommand{\calS}{\mathcal{S}}
\newcommand{\calP}{\mathcal{P}}
\newcommand{\calQ}{\mathcal{Q}}
\newcommand{\e}{\varepsilon}
\newcommand{\Rd}{\mathbb{R}^d}
\newcommand{\OmegaT}{X_{\mathcal{T}}}
\newcommand{\simt}{\sim_{t}}

\newcommand{\LD}{\overset{\mathrm{LD}}{\rightsquigarrow}}
\newcommand{\MLD}{\overset{\mathrm{MLD}}{\leftrightsquigarrow}}

\newcommand{\rhoT}{\rho_{\mathbb{T}}}

\title{Distribution of Patches in Tilings and Spectral Properties of Corresponding Dynamical Systems}
\author[1]{Yasushi Nagai}
\affil[1]{Keio Univ., 3-14-1 Hiyoshi, Kohoku-ku, Yokohama, Kanagawa 223-8522, Japan. (e-mail: nagai@math.keio.ac.jp)}
\date{\today}							

\begin{document}
\maketitle

\begin{abstract}
	A tiling is a cover of $\Rd$ by tiles such as polygons that overlap only on their borders. 
	A patch is a configuration consisting of finitely many tiles that appears in tilings.
	From a tiling, we can construct a dynamical system which encodes the nature of the tiling.
	In the literature, properties of this dynamical system were investigated by studying 
	how patches distribute in each tiling. In this article we conversely research distribution of patches from 
	properties of the corresponding dynamical systems.
	We show periodic structures are hidden in tilings which are not necessarily periodic.
	Our results throw light on inverse problem of deducing information of tilings from information of diffraction measures,
	in a quite general setting.
\end{abstract}

\section{Introduction}
	In 1984 Shechtman et al.\  discovered quasicrystals \cite{shechtman1984metallic}.
	The diffraction patterns of quasicrystals consist of sharp peaks, which show the long-range order of the solid, 
	but have rotational symmetries which crystals can never have. These phenomena 
	required a mathematical explanation and the study of mathematical structures which are non-periodic but
	have some degree of order has been intense since then. Such mathematical structures include for example 
	tilings, Delone sets and translation bounded measures. 
	
	The physical connection between solids and their diffraction images is modeled by considering the Fourier transforms (the diffraction measures) of
	 the autocorrelation measures
	defined by mathematical models such as Delone sets and translation bounded measures.
	There are many researches on this mathematical diffraction: \cite{hof1995diffraction}, \cite{MR2106769}, \cite{MR1798986}, to name a few.
	One of the central questions in the above context is the inverse problem: 
	can we deduce information of the original mathematical objects from information of its diffraction measure?
	(See for example \cite{lenz2011stationary}, \cite{terauds2013inverse}, \cite{terauds2013some}, \cite{grimm2008homometric}.)

	On the other hand, for the study of mathematical objects which show
	 long-range order,
	the study of the corresponding dynamical systems is important.
	(See for example \cite{Solomyak_dynamics}, \cite{So}, \cite{MR2851885}, to name a few.)
	The corresponding dynamical system consists of the closure of the set of all translations of given mathematical object and
	$\Rd$-action on it.
	There are relations between this dynamical system and the diffraction measure.
	For example, a translation bounded measure has a pure point diffraction measure
	if and only if the corresponding dynamical system has pure point spectrum (\cite{MR2106769}), that is,
	the set of eigenfunctions spans a dense subspace of $L^2(\mu)$, $\mu$ being
	an invariant probability measure for the dynamical system.
	
	The diffraction measure typically includes  information on the dynamical system: 
	for example, in \cite{MR2106769} it is shown that the set of measurable eigencharacters (Definition \ref{def_measurable_eigenfunction}) can be seen from 
	the diffraction measure, if the diffraction measure is pure point.
	Furthermore, in some situations all the measurable eigenchracters are topological eigencharacters (Definition \ref{Def-eigenfunction}).
	This is the case for pseudo-self-affine tilings (Theorem \ref{Thm-Solomyak}, Theorem \ref{thm_pseudo_affine}) and 
	repetitive regular model sets (\cite{hof1995diffraction}).
	Thus if we deduce a property on the original mathematical object
	from a property on the set of topological eigenvalues,  that will throw light on
	inverse problem in situations where topological and measurable eigencharacters are the same.
	
	In this article we research properties of tilings of $\mathbb{R}^d$ from properties of the set of topological eigenvalues.
	In the literature, many properties on dynamical system are deduced from the properties on distribution of patches in given tilings.
	The phrase ``distribution of patches'' is not rigorously defined, and may be captured in various ways, for example by
	the number $L(\calP_1, \calP_2)=\card\{x\in\mathbb{R}^d\mid \calP_1+x\subset\calP_2\}$
	 ($\card$ means cardinality, and both $\calP_1$ and $\calP_2$ are patches)
	as follows.
	For example, the corresponding topological dynamical system for a tiling $\calT$ is uniquely ergodic if and only if
	the tiling $\calT$ has uniform patch frequency, that is, for any finite patch $\calP$ there is a limit
	\begin{align*}
		\lim_{n\rightarrow\infty}\frac{L(\calP, \calT\cap A_n)}{m(A_n)}
	\end{align*}
	which is independent of the choice of van Hove sequence $(A_n)$ (\cite{MR1976605}, \cite{Solomyak_dynamics}) .
	Here $m$ is the Lebesgue measure and $\cap$ is defined in Definition \ref{def_cap_sqcap}.
	Also, for a self-affine tiling $\calT$, Solomyak \cite{Solomyak_dynamics} showed that if
	there is $c>0$ such that
	 for any finite patch $\calP$ and a return vector $x$
	we have
	\begin{align*}
		\lim_N \frac{L(\calP\cup(\calP+\phi^n(x)),\calT\cap A_N)}{m(A_N)}\geqq c\frac{L(\calP,\omega^n(T))}{\phi^nm(T)}		
	\end{align*}
	then the corresponding dynamical system is not mixing. 
	(This is the tiling version of Dekking-Keane's result \cite{MR0466485}.)
	For many self-affine tilings, the conditions on distribution of patches
	described by $L$ are first shown, and then the properties of the corresponding dynamical systems are deduced by the above argument.

	In this article we conversely deduce properties on distribution of patches in tilings
	from properties on the dynamical systems. 
	As to distribution of patches, we study it with respect to $\calT$-legality.
	If $\calT$ is a tiling, a patch $\calP$ is said to be $\calT$-legal if a translate of $\calP$ appears somewhere in $\calT$.
	If $\calT$ is repetitive, this is equivalent to $\liminf\frac{L(\calP,\calT\cap A_n)}{m(A_n)}>0$ for some van Hove sequence $(A_n)$.
	As to property of the dynamical systems, as we discussed above, we focus on the set of topological eigenvalues, in particular
	 the property that $0$ is a limit point of the set of eigenvalues.
	For certain class of mathematical models of quasicrystals, for example regular model sets and certain pseudo-self-affine tilings, it can be shown that
	$0$ is a limit point of the set of eigenvalues.
	We deduce, for general repetitive tilings of FLC and of finite tile type (Definition \ref{def_FTT}),
	 from the property that $0$ is a limit point of topological eigenvalues, a property on distribution of patches
	which is described by the term ``$\calT$-legal''.
	For pseudo-self-affine tilings with certain expansion map , more precise conditions are deduced.

	We elaborate. First, for tilings $\calT$
	 of finite tile type, we define a language $\Pi_{\calT,R,\calA}$ for all $R>0$ (Definition \ref{def_language}), which is analogous to the language for sequences 
	 (Definition \ref{def_language_for_sequences}).
	 This set is a collection of all patches in $\calT$ which can be seen from spherical windows of radius $R$.
	 In Theorem \ref{main_thm1}, under the assumption of finite tile type, FLC and repetitivity of tilings,
	we prove the following.
	Assume that $0$ is a limit point of the set of topological eigenvalues of the corresponding topological dynamical system.
	From these assumptions
	we deduce that
	for any bounded neighborhood $U$ of $0\in\mathbb{R}^d$,
	there are $R>0$ and $a\in\mathbb{R}^d\setminus\{0\}$ such that,
	the relative position of copies of two patches $\calP_1$ and $\calP_2$ in $\Pi_{\calT,R,\calA}$ 
	 is not free (Theorem \ref{main_thm1}).
	The relative position is not free in the sense that if we observe a translate of the first patch $\calP_1$ in the tiling then relatively from that position there
	is a periodic region where the second patch $\calP_2$ will never appear.
	The periodic region is a translate of a set $U+\Ker\chi_a$, where $\chi_a(x)=e^{2 \pi i\langle a,x\rangle}$.
	($\langle\cdot,\cdot\rangle$ is the standard inner product.)
	In other words, if $x$ is from that region then $\calP_1\cup(\calP_2+x)$ is not $\calT$-legal.
	The situation is depicted in Figure \ref{figure_main_thm} in page \pageref{figure_main_thm}.
	
%
	
	
	If the tilings are self-affine or more generally pseudo-self-affine (\cite{MR2183221}, \cite{MR1854103}), 
	we can say more (Section \ref{section_substitution_Pisot}).
	 Let $\calT$ be a pseudo-self-affine tiling.
	 Suppose the expansion map for $\calT$ is diagonalizable over $\mathbb{C}$ and all the eigenvalues
	 are algebraic conjugates of the same multiplicity. Suppose also that
	 the spectrum of the map forms a Pisot family.
	We prove that, 
	for any $R_0>0, \e>0$ and finite $F\subset B(0,\frac{1}{8R_0})$, there are $R>0$, $F'\subset \Rd\setminus\{0\}$, and $x_a(\calP)\in\Rd$ for each $a\in F'$ and 
	each patch $\calP\in\Pi_{\calT,R,\calA}$, such that the following conditions hold
	 (Theorem \ref{thm1_substitution}). 
	 First, $F$ and $F'$ are ``close with respect to $\e$'' (Definition \ref{def_epsilon_close}).
	 Second, relative position of two patches $\calP_1,\calP_2\in\Pi_{\calT,R,\calA}$ in $\calT$ is not free
	 in the sense that $(\calP_1+x_a(\calP_1))\cup(\calP_2+x_a(\calP_2)+\frac{1}{2\|a\|^2}a+y+z)$ is not $\calT$-legal for any $y\in\Ker\chi_a, z\in B(0,R_0)$.
	 Thus relative to $\calP_1$, there is a ``forbidden area'' for the appearance of $\calP_2$. The area is obtained by juxtaposing translates of periodic regions
	 $\Ker\chi_a+B(0,R_0)$, $a\in F'$. 
	 
	 By using self-similar structure of pseudo-self-affine tilings, we can see another condition for distribution of patches holds.
	 In Theorem \ref{thm2_substitution}, under the same assumption as above and mild assumption on pseudo-substitution rule,
	 we prove relative position of any two $\calT$-legal patches is not free.
	 The situation is depicted in Figure \ref{figure_introduction} in page
	\pageref{figure_introduction}. The differences from the case of general tilings (Theorem \ref{main_thm1}) in the above paragraph
	are; (1) we are now free from the constraint that the patches $\calP_1,\calP_2$ we
	deal with must be large enough, and (2) the area where $\calP_2$ never appears may be grid-like.

		\begin{figure}[h]
		\begin{center}
		\includegraphics[width=1.0\columnwidth]{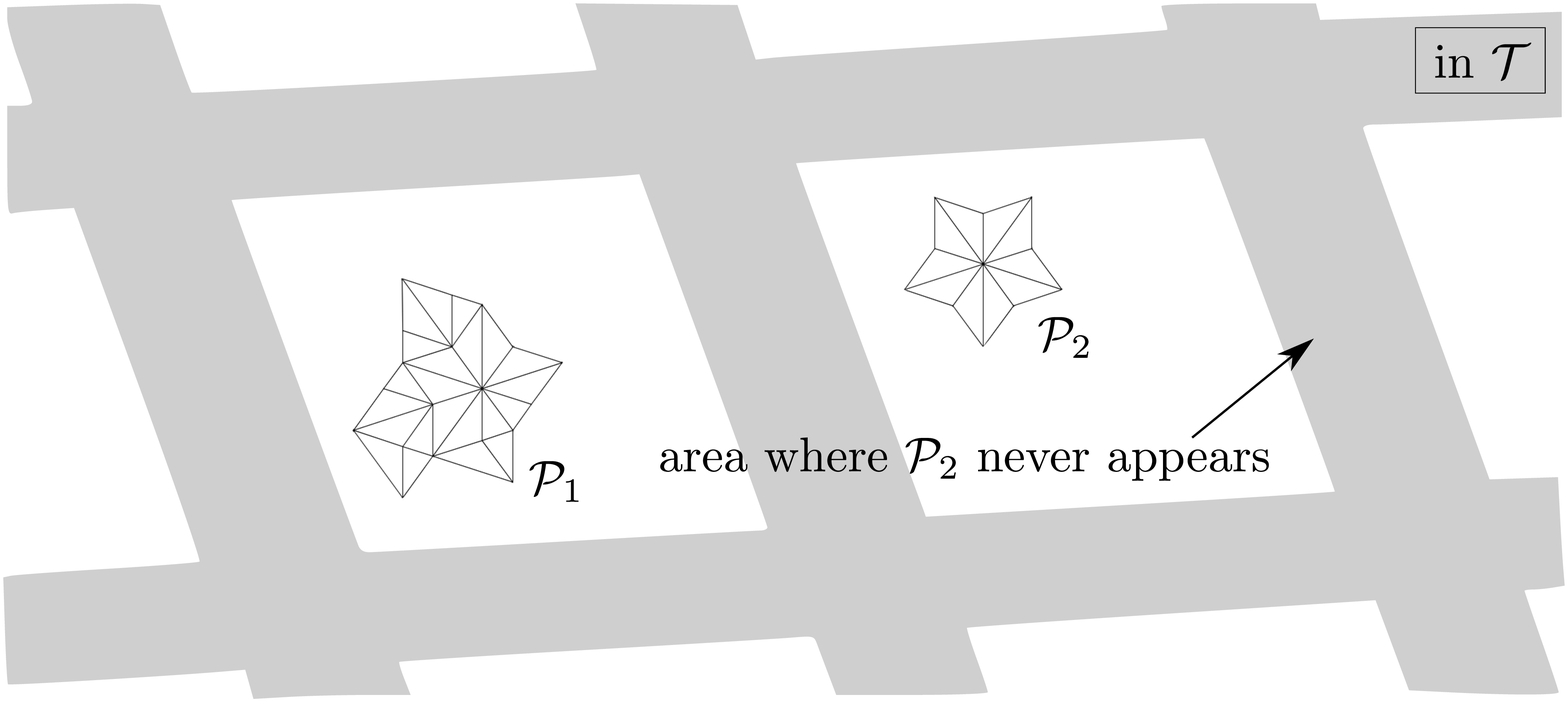}%
		\caption{Theorem \ref{thm2_substitution}}
		\label{figure_introduction}
		\end{center}
		\end{figure}

	On the other hand,
	 in sequences from word substitutions with weakly mixing dynamical systems, 
	 the analogous situation is not repeated. That is,
	 the relative position of words
	is 
	not free.
	See Proposition \ref{prop1_sequence_weak_mixing} and Remark \ref{rem_after_prop1_sequence}.

	Note that in Theorem \ref{main_thm1} and Theorem \ref{thm2_substitution}, periodic structures arise in the area where a patch $\calP_2$ never appears.
	That is, there are hidden symmetries in tilings that are not necessarily periodic (translationally symmetric).
	 Existence of eigenfunctions implies existence of factor maps $(X_{\calT},\Rd)\rightarrow (\mathbb{T},\Rd)$ where
	$\Rd$ acts on $\mathbb{T}=\{z\in\mathbb{C}\mid |z|=1\}$ by $x\cdot z=e^{2\pi i\langle x,a\rangle}z$.
	The dynamical system $(\mathbb{T},\Rd)$ has a wealth of symmetries (there are many $x\in\Rd$ such that $x\cdot z=z$ for any $z\in\mathbb{T}$)
	and these symmetries are propagated to those of $\calT$.

	Theorem \ref{thm2_substitution} is for certain pseudo-self-affine tilings.
	Theorem \ref{main_thm1} is for more general tilings.
	By recognizability (Theorem \ref{thm_recognizability}) and 
	Theorem \ref{thm_lee_solomyak}, Theorem \ref{main_thm1}  is applicable to tilings in the class of non-periodic 
	pseudo-self-affine tilings
	such that the expansion maps are diagonalizable over $\mathbb{C}$, all the eigenvalues are algebraic conjugates
	 of the same multiplicity, and the spectrum of each expansion map forms a Pisot family. 
	We do not rule out the possibility that these can be applicable to tilings in another class, such as Delone tilings for model sets.

	This article is organized as follows. Section \ref{section_def} is for preliminaries.
	Although some of these are known results, we do not omit the proofs for the reader's convenience, because it is hard to extract them from the literature.
	In Section \ref{section_general_tilings}, we deal with general tilings with sufficiently many eigenvalues, under mild assumptions.
	In Section \ref{section_substitution_Pisot} we deal with tilings constructed from tiling substitutions and pseudo-substitutions.
	In Section \ref{section_sequence_weak_mixing} we study sequences with weakly mixing dynamical systems.

\section{Preliminaries}\label{section_def}

	\subsection{Definition of Tiling}
	Throughout the article we set $\Zpo=\{1,2,\ldots\}$ and $\Znn=\{0,1,2,\ldots\}$.
	If $S$ is a set, $\card S$ denotes its cardinality.
	If $(X,\rho)$ is a metric space, $x\in X$ and $r>0$, we write the open ball of radius $r$ with its center $x$ by $B(x,r)$.
	For a metric space $(X,\rho)$ and $S\subset X$, its diameter is by definition $\diam S=\sup\{\rho (x,y)\mid x,y\in S\}$.
	For a topological space $X$ and its subset $A$, its closure, interior, and boundary are denoted by $\overline{A}, A^{\circ}$ and $\partial A$, respectively.

	\begin{defi}
		Take $d\in\Zpo$ and fix it.
		A tile of $\mathbb{R}^d$ is a subset of $\mathbb{R}^d$ which is nonempty, open and bounded.

		A patch of $\mathbb{R}^d$ is a set $\mathcal{P}$ of tiles of $\mathbb{R}^d$ such that, 
		if $S,T\in\mathcal{P}$ and $S\neq T$, then $S\cap T=\emptyset$.

		For a patch $\mathcal{P}$, its support is the subset $\overline{\bigcup_{T\in\mathcal{P}}T}$ of $\mathbb{R}^d$ and denoted by $\supp \mathcal{P}$.
		
		A patch $\mathcal{P}$ is called a tiling if $\supp\mathcal{P}=\mathbb{R}^d$.
	\end{defi}
	
	\begin{rem}\label{rem_definition_tile}
		In the literature, a tile is defined in various ways. For example, it is defined as
		(1) a subset of $\mathbb{R}^d$ which is homeomorphic to a closed unit ball of $\mathbb{R}^d$ (\cite{AP}),
		(2) a closed polygonal subset of $\mathbb{R}^d$ (\cite{Wh}), or
		(3) a subset of $\mathbb{R}^d$ which is compact and equal to the closure of its interior (\cite{BH}).
		
		In all these definitions tiles are defined as compact sets. However we can do the same argument regarding tilings by considering the interiors of tiles.
		This way has one virtue. Often in tiling theory one has to consider labels for tiles, so that one can distinguish tiles of the same shape. If we deal with interiors,
		instead of 
		 compact sets, we can remove different points from two copies of the same tile, so that we can distinguish these two tiles by these ``punctures'' and do not
		need to consider labels separately. (See Example \ref{ex_penrose}).
	\end{rem}
	

	\begin{defi}
		For a patch $\mathcal{P}$ and a vector $x\in\mathbb{R}^d$, define a translate of $\mathcal{P}$ by $x$ via
		$\mathcal{P}+x=\{T+x\mid T\in\mathcal{P}\}$. For two patches $\calP_1$ and $\calP_2$ set $\calP_1\simt\calP_2$ if there is $x\in\Rd$ such that
		$\calP_1+x=\calP_2$.
	\end{defi}

	\begin{defi}
		A tiling $\mathcal{T}$ is said to be sub-periodic if there is $x\in\mathbb{R}^d\setminus\{0\}$ such that its translate by $x$ coincide with itself, that is, 
		$\mathcal{T}+x=\mathcal{T}$. Otherwise a tiling is said to be non-periodic.
		A tiling $\calT$ of $\Rd$ is said to be periodic if there is a basis $\{b_1,b_2,\ldots ,b_d\}$ of $\Rd$ such that $\calT+b_i=\calT$ for all $i$.
	\end{defi}
	
	\begin{ex}[Square tiling]
		For any dimension $d\in\Zpo$, a tiling $\mathcal{T}_s=\{(0,1)^d+v\mid v\in\mathbb{Z}^d\}$ is called Square tiling.
		This is an example of periodic tiling.
	\end{ex}
	
	Many interesting examples of non-periodic tilings can be constructed from substitution rules, which we will introduce later.
	We also note that sequences can be regarded as tilings.
	
	The theme of this article is $\mathcal{T}$-legality.
	\begin{defi}\label{def_legal}
		Given a tiling $\mathcal{T}$, a patch $\mathcal{P}$ is $\mathcal{T}$-legal if there is $x\in\mathbb{R}^d$ such that
		$\mathcal{P}+x\subset\mathcal{T}$.
	\end{defi}

	There are many interesting examples of tilings which have only finitely many types of tiles, up to translation. In this article we confine our interest
	only to such tilings.
	\begin{defi}\label{def_FTT}
		Let $\mathcal{T}$ be a tiling. If there is a finite set $\mathcal{A}$ of tiles such that for any $T\in\mathcal{T}$ there are a unique $P\in\mathcal{A}$ and 
		(necessarily unique) $x\in\mathbb{R}^d$ for which $T=P+x$, then we say $\mathcal{T}$ has finite tile type. $\mathcal{A}$ is called a set of proto-tiles
		for $\mathcal{T}$.
		If $\mathcal{T}$ is a tiling of finite tile type, then we always take a set of proto-tiles $\mathcal{A}$ such that $0\in P$ for each $P\in\mathcal{A}$.
	\end{defi}
	
	\begin{defi}\label{def_language}
		For $R>0$ and a tiling $\mathcal{T}$ of finite tile type with a set of proto-tiles $\mathcal{A}$, set
		\begin{align*}
			\Pi_{\mathcal{T},R,\calA}=\{(\mathcal{T}-x)\cap B(0,R)\mid x\in\mathbb{R}^d, (\mathcal{T}-x)\cap\mathcal{A}\neq\emptyset\}.
		\end{align*}
		Also set $\Pi_{\calT,\calA}=\{\calP\mid\text{$\calP$ is a $\calT$-legal patch and $\calP\cap\mathcal{A}\neq\emptyset$.}\}$.
	\end{defi}	
	
	\begin{rem}
		This is a tiling-version of language for sequences (Definition \ref{def_language_for_sequences}).
	\end{rem}
	
	Next we introduce an important concept in tiling theory, which is called FLC (Definition \ref{def_FLC}). We prove two characterizations of FLC 
	(Lemma \ref{lem_FLC_iikae1} and Proposition \ref{FLC_iikae}).
	\begin{defi}\label{def_cap_sqcap}
		For a patch $\calP$ and a subset $S\subset \Rd$, set $\calP\cap S=\{T\in\calP\mid T\subset S\}$ and 
		$\calP\sqcap S=\{T\in\calP\mid \overline{T}\cap S\neq\emptyset \}$.
	\end{defi}
	
	\begin{defi}\label{def_FLC}
		A tiling $\mathcal{T}$ has finite local complexity (FLC) if for any $R>0$ the set $\{\mathcal{T}\cap B(x,R)\mid x\in\mathbb{R}^d\}/{\simt}$ is finite.
	\end{defi}
	
	\begin{lem}\label{lem_FTT_sqcap_finite}
		Let $\calT$ be a tiling of finite tile type and $S$ be a bounded subset of $\Rd$. Then there is an integer $N=N_{\calT, S}$ such that
		$\card\calT\sqcap (S+x)\leqq N$ for any $x\in\Rd$.
	\end{lem}
	\begin{proof}
		Set $r=\max_{T\in\calT}\diam T$. Take $R>0$ such that $S\subset B(0,R)$. For any $x\in\Rd$, $\supp(\calT-x)\sqcap S$ is contained in $B(0,R+r)$.
		There is $s>0$ such that any tile in $\calT$ contains a ball of radius $s$. There is an integer $N$ such that $B(0,R+r)$ contains at most 
		$N$ mutually disjoint balls of radius $s$. It follows that $(\calT-x)\sqcap S$ contains at most $N$ tiles.
	\end{proof}

	\begin{lem}\label{lem_FLC_iikae1}
		For a tiling $\calT$ of finite tile type, the following conditions are equivalent:
		\begin{enumerate}
			\item $\calT$ has FLC.
			\item $\{\calT\sqcap B(x,R)\mid x\in\Rd\}/{\simt}$ is finite for all $R>0$.
		\end{enumerate}
	\end{lem}
	
	Note that if $r=\max_{T\in\calT}\diam T$, then $\calT\sqcap B(x,R)\subset\calT\cap B(x,R+r)$. The proof for Lemma \ref{lem_FLC_iikae1} is straightforward
	if we use Lemma \ref{lem_FTT_sqcap_finite} and the following lemma.
	\begin{lem}\label{lem_family_patch_finite}
		Let $\Pi_0,\Pi_1$ be families of patches such that (1) any $\calP\in\Pi_1$ is finite, (2) for any $\calP_0\in\Pi_0$ there are
		$\calP_1\in\Pi_1$ and  $x\in\Rd$ for which $\calP_0+x\subset\calP_1$, and (3) $\Pi_1/{\simt}$ is finite.
		Then $\Pi_0/{\simt}$ is finite.
	\end{lem}

	The following another characterization of FLC will be useful. Recall Definition \ref{def_language}.
	
	\begin{prop}\label{FLC_iikae}
		Let $\mathcal{T}$ be a tiling of finite tile type with a set of proto-tiles $\mathcal{A}$. Then the following conditions are
		equivalent:
		\begin{enumerate}
			\item $\mathcal{T}$ has FLC.
			\item $\Pi_{\calT,R,\calA}$ is finite for all $R>0$.
		\end{enumerate}
	\end{prop}
	
	\begin{proof}
		(1)$\Rightarrow$ (2): 
		It is enough to show that $\Pi_{\calT,R,\calA}$ is finite for any $R>\max_{P\in\mathcal{A}}\diam P$ since
		there is a surjection $\Pi_{\calT,R,\calA}\rightarrow\Pi_{\calT,R',\calA}$ for $R'<R$.
		
		Suppose $R>\max_{P\in\mathcal{A}}\diam P$.
		It suffices to show for each $\calP\in\Pi_{\calT,R,\calA}$, the set $Z_{\calP}=\{z\in\Rd\mid \calP+z\in\Pi_{\calT,R,\calA}\}$ is finite since 
		$\Pi_{\calT,R,\calA}/{\simt}$ is
		finite by (1). Take $\calP\in\Pi_{\calT,R,\calA}$. Define a map $\varphi\colon Z_{\calP}\rightarrow\calP$ as follows. For $z\in Z_{\calP}$, there is a unique
		$P\in(\calP+z)\cap\mathcal{A}$. In fact, there are
		 $y\in\Rd$ and $P\in\mathcal{A}$ such that $\calP+z=(\calT-y)\cap B(0,R)$ and $P\in (\calT-y)\cap\mathcal{A}$.
		Since we assumed $R$ was large enough, $P\in\calP +z$. Set $\varphi(z)=P-z$. This map $\varphi$ is injective. In fact, if $\varphi(z_1)=\varphi(z_2)$,
		take $P_i\in(\calP+z_i)\cap\mathcal{A}$. Then $P_1-z_1=P_2-z_2$ and so $P_1=P_2$ and $z_1=z_2$. Since $\calP$ is finite, we see $Z_{\calP}$ is finite.

		(2)$\Rightarrow$(1):
		Take $R>0$. Set $\Pi_1=\{(\calT-x)\cap B(0,R)\mid x\in\Rd\}$ and we show $\Pi_1/{\simt}$ is finite. Take $L>\max \{R,\max_{P\in\mathcal{A}}\diam P\}$ 
		and set
		$\Pi_2=\{(\calT-x)\cap B(0,L)\mid x\in\Rd\}$. By Lemma \ref{lem_family_patch_finite} and Lemma 
		\ref{lem_FTT_sqcap_finite}, it suffices to show $\Pi_2/{\simt}$ is finite. Take
		$(\calT-x)\cap B(0,L)\in\Pi_2$. Choose $P\in\mathcal{A}$ and $y\in\Rd$ such that $P+y\in (\calT-x)\cap B(0,L)$. Then $\|y\|<L$ and 
		$P\in\calT-x-y$. We have
		\begin{align*}
			(\calT-x)\cap B(0,L)-y=(\calT-x-y)\cap B(-y,L)\subset (\calT-x-y)\cap B(0,2L).
		\end{align*}
		We have proved that, for any $\calP\in\Pi_2$ there are $y\in\Rd$ and $\calP'\in\Pi_{\calT,2L,\calA}$ such that $\calP-y\subset\calP'$.
		$\Pi_{\calT,2L,\calA}/{\simt}$
		 is finite since $\Pi_{\calT,2L,\calA}$ is finite by assumption. Therefore by Lemma \ref{lem_family_patch_finite}, $\Pi_2/{\simt}$ is finite.
	\end{proof}

	\subsection{Tiling Space and Tiling Dynamical Systems}
	To study the nature of tilings, researching corresponding continuous hulls and tiling dynamical systems is useful. Let $\|\cdot\|$ be the 
	standard $2$-norm of $\mathbb{R}^d$. 

	\begin{defi}
		The set of all patches of $\mathbb{R}^d$ is denoted by $\Omega(\mathbb{R}^d)$.
	\end{defi}

	First we define a metric on $\Omega(\mathbb{R}^d)$, which is based on a well-known idea: to regard two patches close if,
	 after small translation,
	they precisely agree on a large ball about the origin.

	Recall that for a patch $\mathcal{P}$ and $S\subset\mathbb{R}^d$, we set $\mathcal{P}\cap S=\{T\in\mathcal{P}\mid T\subset S\}$.
	For two patches $\mathcal{P}_1,\mathcal{P}_2$ of $\mathbb{R}^d$, set
	
	\begin{align*}
		\Delta(\mathcal{P}_1, \mathcal{P}_2)=\biggl\{0<r<\frac{1}{\sqrt{2}}\mid &\text{there exist $x, y\in B(0,r)$ such that}\\
		&(\mathcal{P}_1+x)\cap B(0,\frac{1}{r})=(\mathcal{P}_2+y)\cap B(0,\frac{1}{r})\biggr\}.
	\end{align*}
	
	Then define
	\begin{align}\label{def_metric}
		\rho(\mathcal{P}_1, \mathcal{P}_2)=\inf(\Delta(\mathcal{P}_1, \mathcal{P}_2)\cup\{\frac{1}{\sqrt{2}}\}).
	\end{align}

	\begin{rem}
		It is tempting in the definition of tiling metric to replace $\Delta(\calP_1,\calP_2)$ above with 
		\begin{align*}
			\{0<r<\frac{1}{\sqrt{2}}\mid \text{there exists $y\in B(0,r)$ such that 
		$\mathcal{P}_1\cap B(0,\frac{1}{r})=(\mathcal{P}_2+y)\cap B(0,\frac{1}{r})$}\}
		\end{align*}
		because this definition seems to simplify the following proofs. However if we define the function $\rho$ in this way $\rho$ does not become a metric;
		it is not necessarily true that $\rho(\calT_1,\calT_2)=\rho(\calT_2,\calT_1)$ for two tilings $\calT_1$ and $\calT_2$. Here is an easy counterexample;
		take small $r>0$, and consider three copies of a tile $(-1,1)^d$. Give each of them a puncture in three different ways so that we obtain
		three different tiles (or equivalently, put three different labels to each of the copies so that we can distinguish them). Let
		$S,T,U$ denote the three tiles. Set 
		$\calT_1=\{S\}\cup(\{T\}+2\mathbb{Z}^d\setminus\{0\})$ and $\calT_2=(\{S\}\cup(\{U\}+2\mathbb{Z}^d\setminus\{0\}))+(r,0,0,\ldots 0)$. Then
		$\rho(\calT_1,\calT_2)=\frac{1}{\sqrt{8+d}}$ and $\rho(\calT_2,\calT_1)=\frac{1}{\sqrt{(3-r)^2+d-1}}$.
	\end{rem}
	
	
	It is easy to prove that $\rho$ in (\ref{def_metric}) is a metric on $\Omega(\mathbb{R}^d)$. To prove 
	$\rho(\calT_1,\calT_2)=0$ implies $\calT_1=\calT_2$, we use the following lemma.
	
	\begin{lem}
		Let $T$ be a tile and $\calP$ be a patch. Suppose $x_1,x_2,\cdots$ are elements of $\Rd$ such that 
		$x_n\rightarrow 0$ as $n\rightarrow\infty$ and $T+x_n\in\calP$ for all $n$. Then $T\in\calP$.
	\end{lem}
	
	To prove the triangle inequality, one has to use the fact that $\frac{1}{\e}-\eta>\frac{1}{\e+\eta}$ whenever $0<\e<\frac{1}{\sqrt{2}}$ and 
	$0<\eta<\frac{1}{\sqrt{2}}$.

	\begin{prop}\label{prop_complete}
		The metric space $(\Omega(\mathbb{R}^d), \rho)$ is complete.
	\end{prop}
	\begin{proof}
		Take a Cauchy sequence $(\mathcal{P}_n)_{n\in\Zpo}$.  We may assume
		\begin{align*}
			\rho(\calP_n,\calP_{n+1})<\frac{1}{2^n}
		\end{align*}
		holds for each $n\in\Zpo$. For any $n\in\Zpo$ there are $x_n,y_n\in B(0,\frac{1}{2^n})$ such that
		\begin{align*}
			(\calP_n+x_n)\cap B(0,2^n)=(\calP_{n+1}+y_n)\cap B(0,2^n).
		\end{align*}
		Set $z_n=\sum_{k=n}^{\infty}(x_k-y_k)$. Then $\|z_n\|<\frac{1}{2^{n-2}}$. Set
		\begin{align*}
			\calQ_n=(\calP_n\cap B(0,2^n-1))+z_n	
		\end{align*}
		for each $n\in\Zpo$. For any $n\in\Zpo$ we have $Q_n\subset Q_{n+1}$ and $\calP=\bigcup_{n=1}^{\infty}\calQ_n$ is a patch.
		 Also, one can show that
		if $n<m$, $\calQ_{m+1}\cap B(0,2^n)=\calQ_m\cap B(0,2^n)$.
		 From this
		we can show
		\begin{align*}
			\calP\cap B(0,2^{n-1})=(\calP_{n}+z_n)\cap B(0,2^{n-1}).
		\end{align*} 
		and $\calP=\lim\calP_n$.
	\end{proof}	
	
	\begin{defi}
		For a tiling $\mathcal{T}$, its continuous hull is $X_{\mathcal{T}}=\overline{\{T+x\mid x\in\mathbb{R}^d\}}$ (the closure in 
		$\Omega(\mathbb{R}^d)$ with respect to the tiling metric defined above).
	\end{defi}

	\begin{rem}
		For any tiling $\calT$, $\calS\in X_{\calT}$ and bounded $\calP\subset\calS$, there is $x\in\Rd$ such that $\calP+x\subset \calT$.
	\end{rem}
	If a tiling has FLC, we have the following result:
		
%

	\begin{prop}
		If a tiling $\mathcal{T}$ has FLC,  then its continuous hull $X_{\mathcal{T}}$ is compact.
	\end{prop}
	\begin{proof}
		By Proposition \ref{prop_complete}, it suffices to show that $X_{\mathcal{T}}$ is totally bounded.
		Take $\e>0$ arbitrarily and we shall prove the existence of $\e$-net of $\OmegaT$.
		If $\e\geq\frac{1}{\sqrt{2}}$ then $B(\calT, \e)=\OmegaT$. We may assume $\e<\frac{1}{\sqrt{2}}$.
		Take $\eta<\e$ such that $\frac{1}{\eta}-\eta>\frac{1}{\e}$. By FLC there is a finite set $\Pi$ of patches included in $B(0,\frac{1}{\eta})$ such that
		if $\calS\in\OmegaT$, there are $\calP\in\Pi$ and $y\in\Rd$ for which $\calS\cap B(0,\frac{1}{\eta})=\calP+y$.
		For any $\calP\in\Pi$ set
		\begin{align*}
			Y_{\calP}=\{y\in B(0,\frac{2}{\eta})\mid \text{there is $\calS\in\OmegaT$ such that $\calS\cap B(0,\frac{1}{\eta})=\calP+y$}\}.
		\end{align*}
		Since $Y_{\calP}$ is a bounded set, we can take its finite $\eta$-net $Z_{\calP}$.
		For any $z\in Z_{\calP}$, there is $\calS_{z,\calP}\in\OmegaT$ such that $\calS_{z,\calP}\cap B(0,\frac{1}{\eta})=\calP+z$.
		We claim $\OmegaT\subset\bigcup_{\calP\in\Pi}\bigcup_{z\in Z_{\calP}}B(\calS_{z,\calP},\e)$.
		 If $\mathcal{S}\in X_{\calT}$, then there are $\calP\in\Pi$ and $y\in\Rd$ such that $\mathcal{S}\cap B(0,\frac{1}{\eta})=\calP+y$.
		 If $\calP=\emptyset$, we can take $y=0$. Otherwise we have $y\in B(0,\frac{2}{\eta})$ because
		  $T, T+y\subset B(0,\frac{1}{\eta})$ for any $T\in\calP$.
		 In both cases we can take $z\in Z_{\calP}$ such that $\|z-y\|<\eta$. We have 
		 $\mathcal{S}\cap B(0,\frac{1}{\e})=(\mathcal{S}_{z,\calP}+y-z)\cap B(0,\frac{1}{\e})$ and $\rho(\mathcal{S},\mathcal{S}_{z,\calP})<\e$.
	\end{proof}

%

	We then introduce cylinder sets, which form a basis for the relative topology of the metric topology on certain continuous hulls.

	\begin{defi}
		Take a tiling $\mathcal{T}$, a $\mathcal{T}$-legal patch $\mathcal{P}$ and an open neighborhood $U$ of $0$ in $\mathbb{R}^d$. Set
		\begin{align*}
			C_{\mathcal{T}}(U,\mathcal{P})=\{\mathcal{S}\in X_{\mathcal{T}}\mid \text{there is $x\in U$ such that $\mathcal{P}+x\subset\mathcal{S}$}\}.
		\end{align*}
	\end{defi}

	We can show that if $\calP$ is finite $\calT$-legal and $U$ is an open neighborhood of $0\in\Rd$, then $C_{\calT}(U,\calP)$ is open
	with respect to the relative topology of the metric topology on $X_{\calT}$.
	If $\calT$ has finite tile type, then 
	\begin{align*}
			\{C_{\mathcal{T}}(U,\mathcal{P})\mid \text{$\mathcal{P}$: finite $\mathcal{T}$-legal patch and $U$: open neighborhood of $0$ in $\mathbb{R}^d$}\}
	\end{align*}
		generates the relative topology of metric topology on $X_{\mathcal{T}}$.

	The continuous hull of any tiling $\mathcal{T}$ has an $\mathbb{R}^d$-action via translation:
	\begin{align*}
		X_{\mathcal{T}}\times \mathbb{R}^d\ni (\mathcal{S}, x)\mapsto \mathcal{S}-x\in X_{\mathcal{T}}.
	\end{align*}
	This is jointly continuous. The topological dynamical system $(X_{\calT},\Rd)$ is called the tiling dynamical system associated to $\mathcal{T}$.
	
	Next we introduce a sufficient condition for the tiling dynamical system to be minimal. Recall that a dynamical system $\mathbb{R}^d\curvearrowright \Omega$,
	where $\Omega$ is a compact Hausdorff space, is minimal if any of its orbits is dense in $\Omega$. 
		
	\begin{defi}
		A subset $S\subset\mathbb{R}^d$ is relatively dense in $\mathbb{R}^d$
		if there is $R>0$ such that for any $x\in\mathbb{R}^d$, we have $B(x,R)\cap S\neq\emptyset$.
	\end{defi}
	
	\begin{defi}\label{def_repetitive}
		A tiling $\mathcal{T}$ is repetitive if for any finite patch $\mathcal{P}\subset\mathcal{T}$, the set $\{x\in\mathbb{R}^d\mid \mathcal{P}+x\subset\mathcal{T}\}$
		is relatively dense in $\mathbb{R}^d$.
	\end{defi}
	\begin{rem}
		In the literature repetitivity is defined in various ways. For example, a tiling $\calT$ may be defined to be repetitive if for any
		compact $K\subset\Rd$, there is a compact set $K'\subset\Rd$ such that whenever $x_1,x_2\in\Rd$ there is $y\in K'$ with
		$(\calT+x_1)\cap K=(\calT+x_2+y)\cap K$. For FLC tilings of finite tile type, this condition and our definition of repetitivity are equivalent.
		Definition \ref{def_repetitive} also coincides with the definition in \cite{So}.
	\end{rem}
	
	\begin{lem}
		Let $\calT$ be a tiling of $\Rd$ of finite tile type.
		If $\mathcal{T}$ is repetitive, then the associated tiling dynamical system is minimal.
	\end{lem}
	\begin{proof}
		Take $\calS\in X_{\calT}$.
		Let  $\calP$ be a finite $\calT$-legal patch. 
		It suffices to show that a translate of $\calP$ appears in $\calS$, since if so
		a translate of $\calS$ intersects with $C_{\calT}(U,\calP)$ for any $U$.
		By repetitivity, there is $R>0$ such that
		$\{B(\lambda,R)\mid\lambda\in\Lambda\}$ cover $\Rd$, where
		$\Lambda=\{x\in\Rd\mid \calP+x\subset\calT\}$.
		Take $\e>0$ such that if $\|x\|<R$, then $\supp\calP+x\subset B(0,\frac{1}{\e})$. 
		There is $y\in\Rd$ such that $\rho(\calS, \calT+y)<\e$. $\calS$ and $\calT+y$ almost agree on a large ball
		around the origin, and on this large ball a copy of $\calP$ must appear. 
	\end{proof}

        \subsection{Eigenfunctions for dynamical systems}
        Recall that for a locally compact abelian group $G$ and $\mathbb{T}=\{z\in\mathbb{C}\mid |z|=1\}$, a continuous group homomorphism 
        $\chi\colon G\rightarrow \mathbb{T}$ is called a character.
	\begin{defi}\label{Def-eigenfunction}
		Let $G$ be a locally compact abelian group and $X$ be a compact space. Assume $G$ acts on $X$ continuously. 
		Then we call a continuous function $f\colon X\rightarrow \mathbb{C}$ a continuous eigenfunction if
		$f\neq 0$ and 
		there is a character $\chi\colon G\rightarrow \mathbb{T}$ such that $f(g\cdot x)=\chi(g)f(x)$ holds for any $g\in G$ and $x\in X$.
		We call this character $\chi$ the (topological) eigencharacter.
	\end{defi}

	\begin{defi}\label{def_measurable_eigenfunction}
		Let $(X,\mathcal{B},\mu)$ be a probability space and suppose a locally compact abelian group $G$ acts on $X$ in a measure-preserving way.
		Then $f\in L^2(\mu)$ is called an $L^2$-eigenfunction if $f\neq 0$ and there is a character $\chi$ such that 
		$U_g(f)=\chi(g)f$ for all $g\in G$, where $U_g$ is the corresponding unitary operator on $L^2(\mu)$.
		We call this $\chi$ the (measurable) eigencharacter.
	\end{defi}
	
	\begin{defi}
		For $a\in\mathbb{R}^d$, let $\chi_a$ be the character of $\Rd$ defined by $\chi_a(x)=e^{2\pi i\langle a,x\rangle}$, 
		where $\langle\cdot ,\cdot\rangle$ is the standard inner product of $\mathbb{R}^d$.
		 (Every character of $\mathbb{R}^d$ is of this form.)
		If $G=\mathbb{R}^d$ in Definition \ref{Def-eigenfunction} (resp.\  in Definition \ref{def_measurable_eigenfunction}),
		 and $\chi_a$ is the eigencharacter of a continuous 
		eigenfunction $f$ (resp.\  $L^2$-eigenfunction $f$), $a$ is called a topological eigenvalue (resp.\  measurable eigenvalue) for 
		the topological dynamical system $(X, \mathbb{R}^d)$ (resp.\ measure-preserving system $(X,\mathcal{B}, \mu, \mathbb{R}^d)$). 
	\end{defi}
	
	\begin{rem}
		If $(X,\mathbb{R}^d)$ is a uniquely ergodic topological dynamical system with a
		unique invariant measure $\mu$, in some cases
		for any  measurable eigenvalue $a$, we can always take an eigenfunction $f\in L^2(\mu)$ which is continuous.
		Thus in this situation topological eigenvalue and measurable eigenvalue are equivalent concepts
		and we say simply ``eigenvalue'' for topological and measurable eigenvalue.
		See Theorem \ref{Thm-Solomyak}.
	\end{rem}

	\begin{rem}
		If $X$ is compact and $G$ acts on $X$ minimally, for any eigencharacter $\chi$ we can always take
		a continuous eigenfunction $f$ which takes value in $\mathbb{T}$. In later sections all the topological dynamics are minimal,
		and we tacitly use this fact about eigenfunctions.
	\end{rem}

\section{General Tilings with Sufficiently Many Eigenvalues}\label{section_general_tilings}

	Let $\calT$ be a tiling of $\Rd$ of finite tile type with a set of proto-tiles $\mathcal{A}$.
	Recall Definition \ref{def_language}.
	 By
	Proposition \ref{FLC_iikae}, $\Pi_{\calT,R,\mathcal{A}}$ is finite  for all $R>0$ if and only if the tiling $\mathcal{T}$ has FLC.

	Recall also that we set $\chi_a(x)=e^{2\pi i\langle x,a\rangle}$.
	
	\begin{defi}
		We endow a metric $\rhoT$ on $\mathbb{T}$ by identifying $\mathbb{T}$ with $\mathbb{R}/2\pi\mathbb{Z}$. In other words we set 
		\begin{align*}
			\rhoT(e^{2\pi i\theta}, e^{2\pi i \theta'})=\min_{n\in\mathbb{Z}}|\theta-\theta'+n|
		\end{align*}
		for any $\theta, \theta'\in\mathbb{R}$. This gives a well-defined metric on $\mathbb{T}$ that generates the standard topology of $\mathbb{T}$.
	\end{defi}
	
	\begin{setting}
	In this section we assume $\calT$ is a repetitive tiling of $\Rd$ of FLC of finite tile type, with a set of proto-tiles $\calA$.
	\end{setting}
	
	\begin{lem}\label{Lem1_main_thm}
		Let $a\in\Rd\setminus\{0\}$ be a topological
		 eigenvalue for the topological dynamical system $(X_{\calT}, \Rd)$ with a continuous $\mathbb{T}$-valued 
		eigenfunction $f$.
		Also let $\e$ be a positive real number and $U$ be a bounded subset of $\Rd$.
		Assume $\sup_{x\in U}\rhoT(\chi_a(x),1)<\frac{\e}{2}$.
		Then there exists a positive real number $R$ such that
		 $\mathcal{P}\in\bigcup_{L>R}\Pi_{\mathcal{T},L,\calA}$ and 
		$\mathcal{S}_1,\mathcal{S}_2\in C_{\calT}(U,\mathcal{P})$ imply $\rhoT(f(\mathcal{S}_1),f(\mathcal{S}_2))<\e$.
	\end{lem}

	\begin{proof}
		Set $\eta=\sup_{x\in U}\rhoT(\chi_a(x),1)$.
		Since $X_{\calT}$ is compact and $f$ is uniformly continuous, there is $\delta>0$ such that $\rho(\mathcal{S}_1,\mathcal{S}_2)<\delta$ implies 
		$\rhoT(f(\mathcal{S}_1),f(\mathcal{S}_2))<\e-2\eta$.
		
		Set $R=\frac{1}{\delta}+\max_{T\in\mathcal{T}}\diam T$. This $R$ has the desired property.
	\end{proof}

	\begin{lem}\label{Lem2_main_thm}
		Let $U$ be a neighborhood of 0 in $\mathbb{R}^d$, $\mathcal{P}$ and $\mathcal{P}'$ $\mathcal{T}$-legal patches, and
		$f$ a continuous eigenfunction for the action $\mathbb{R}^d\curvearrowright X_{\mathcal{T}}$ with an eigenvalue $a$.
		Suppose
		\begin{align}\label{condition_empty}
			f(C_{\calT}(U,\mathcal{P}))\cap f(C_{\calT}(U,\mathcal{P}'))=\emptyset.
		\end{align}
		Then $\mathcal{P}\cup(\mathcal{P}'+y+z)$ is \emph{not} $\mathcal{T}$-legal for any $y\in\Ker\chi_a$ and $z\in U$.
	\end{lem}

	\begin{proof}
		Take $\mathcal{S}\in X_{\mathcal{T}}$ such that $\mathcal{P}\subset\mathcal{S}$ and it suffices to prove that
		$\mathcal{P}'+y+z\not\subset\mathcal{S}$. 
		For such $\mathcal{S}$ we have $\mathcal{S}\in C_{\calT}(U,\mathcal{P})$. By condition (\ref{condition_empty}), for any $y\in\Ker\chi_a$,
		\begin{align*}
			f(\mathcal{S}-y)=f(\mathcal{S})\notin f(C_{\calT}(U,\mathcal{P}')).
		\end{align*}
		Then $\mathcal{S}-y\notin C_{\calT}(U,\mathcal{P}')$, and by the definition of cylinder set, the desired condition is deduced.
	\end{proof}
	
	Recall that if $X$ is a metric space, the open ball with its center $x$ and its radius $R>0$ is denoted by $B(x,R)$.

	\begin{prop}\label{prop_main_thm}
		Let $a\in\Rd\setminus\{0\}$ be a topological eigenvalue for the topological dynamical system $(X_{\calT},\Rd)$
		with a continuous $\mathbb{T}$-valued eigenfunction $f$.
		Also let $U$ be a bounded neighborhood for $0\in\Rd$.
		Assume $\sup_{x\in U}\rhoT(\chi_a(x),1)<\frac{1}{8}$.
		Then there exist $R>0$ and a map $x:\bigcup_{L>R}\Pi_{\calT,L,\calA}\rightarrow\mathbb{R}^d$ such that
		the patch
		\begin{align*}
			(\calP_1+x(\calP_1))\cup (\calP_2+x(\calP_2)+\frac{1}{2\| a\|^2}a+y+z)
		\end{align*}
		is \emph{not} $\calT$-legal for any $\calP_1,\calP_2\in\bigcup_{L>R}\Pi_{\mathcal{T},L,\calA}$, $y\in\Ker\chi_a$ and $z\in U$.
	\end{prop}
	
	\begin{proof}
		For $a, \e=\frac{1}{4}$ and $U$, apply Lemma \ref{Lem1_main_thm}. We obtain a positive real number $R$ such that
		$\mathcal{P}\in\bigcup_{L>R}\Pi_{\mathcal{T},L,\calA}$ and 
		$\mathcal{S}_1,\mathcal{S}_2\in C_{\calT}(U,\mathcal{P})$ imply $\rhoT(f(\mathcal{S}_1),f(\mathcal{S}_2))<\frac{1}{4}$.
		For each $\calP\in\bigcup_{L>R}\Pi_{\calT,L,\calA}$, take $\calS\in C_{\calT}(U,\calP)$ and choose $x(\calP)\in\Rd$ such that
		$\chi_a(x(\calP))=f(\calS)$.
		If $\calP_1,\calP_2\in\bigcup_{L>R}\Pi_{\mathcal{T},L,\calA}$, then $f(C_{\calT}(U,\calP_1+x(\calP_1)))\subset B(1,1/4)$ and 
		$f(C_{\calT}(U,\calP_2+x(\calP_2)+\frac{1}{2\| a\|^2}a))\subset B(-1,1/4)$ 
		and so $f(C_{\calT}(U,\calP_1+x(\calP_1)))\cap f(C_{\calT}(U,\calP_2+x(\calP_2)+\frac{1}{2\| a\|^2}a))=\emptyset$. 
		Lemma \ref{Lem2_main_thm} applies.
	\end{proof}
	
	\begin{thm}\label{main_thm1}
		Let $A$ be a subgroup of the group of all topological eigenvalues for the topological dynamical system
		$(X_{\mathcal{T}},\Rd)$. Assume $0\in\overline{A\setminus\{0\}}$.
		Then for any $R_0>0$ and $\e>0$, there exist $R>0$, $a\in A\setminus\{0\}$,
		and $x:\bigcup_{L>R}\Pi_{\calT,L,\calA}\rightarrow\mathbb{R}^d$  
		that satisfy the following two conditions:
		\begin{itemize}
			\item for any
				$\mathcal{P}_1,\mathcal{P}_2\in\bigcup_{L>R}\Pi_{\mathcal{T},L,\calA}$, $y\in\Ker\chi_a$ and $z\in B(0,R_0)$, 
				the patch
				\begin{align*}
					(\calP_1+x(\calP_1))\cup (\calP_2+x(\calP_2)+\frac{1}{2\| a\|^2}a+y+z)
				\end{align*}
				is \emph{not} $\calT$-legal.
			\item \begin{align}
					8R_0<\frac{1}{\|a\|}<(8+\e)R_0.\label{eq_relation_band_interval}
				\end{align}
		\end{itemize}
	\end{thm}
	\begin{proof}
		We can take $a\in A\setminus\{0\}$ such that $\frac{1}{(8+\e)R_0}<\|a\|<\frac{1}{8R_0}$.
		Proposition \ref{prop_main_thm} is applicable for this $a$ and $U=B(0,R_0)$.
	\end{proof}

		\begin{figure}[h]
		\begin{center}
		\includegraphics[width=1.0\columnwidth]{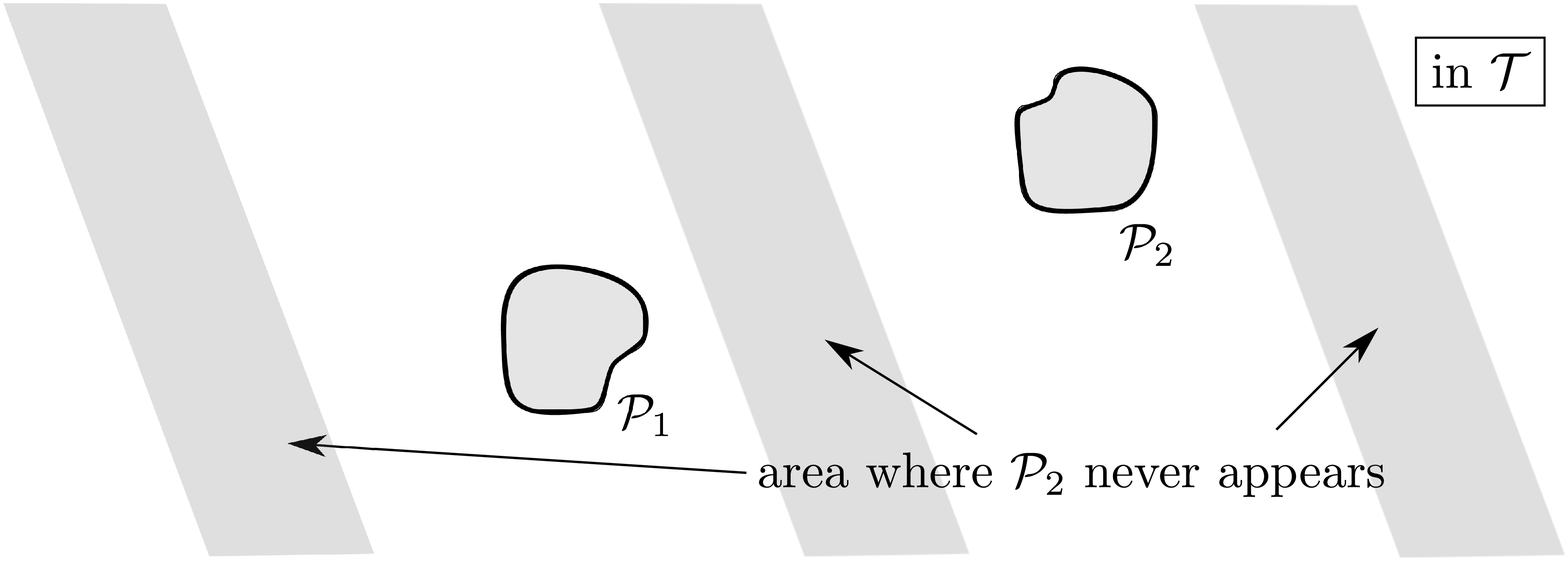}%
		\caption{Theorem \ref{main_thm1}}
		\label{figure_main_thm}
		\end{center}
		\end{figure}

	\begin{rem}	
		In Theorem \ref{main_thm1}, for any $R_0$ and $\e>0$, there are $R>0$ and $a\in A$ such that for any $\calP_1, \calP_2\in\bigcup_{L>R}\Pi_{\calT,L,\calA}$, 
		if we observe a translate of 
		$\calP_1$ in $\calT$, the area where copies of $\calP_2$ do not appear relative to this translate of $\calP_1$ includes a translate of $B(0,R_0)+\Ker\chi_a$
		(Figure \ref{figure_main_thm}).
		Thus there is a ``forbidden area'', the shape of which is described as follows.
		The kernel $\Ker\chi_a$ is equal to $\{a\}^{\perp}+\mathbb{Z}\frac{1}{\|a\|^2}a$.  The set $B(0,R_0)+\Ker\chi_a$ can be obtained by juxtaposing 
		translates of
		a ``band''
		$\{a\}^{\perp}+B(0,R_0)$ by vectors in $\mathbb{Z}\frac{1}{\|a\|^2}a$. The inequality \eqref{eq_relation_band_interval} describes the relation between 
		the width of each band and the interval of bands. (The width of each band is $2R_0$ and the intervals are approximately $8R_0$
		with respect to $\e$.)
	\end{rem}
	
	Next we proceed to Theorem \ref{main_thm2}. Before proving it we prepare a lemma and a definition.

	\begin{lem}
		If $U$ is a subset of $\mathbb{R}^d$ such that $-\overline{P}\subset U$ for all $P\in\mathcal{A}$, then
		$\{C_{\calT}(U,\mathcal{P})\mid \mathcal{P}\in\Pi_{\mathcal{T},R,\calA}\}$ covers $X_{\mathcal{T}}$ for any $R>0$.
	\end{lem}
	
	\begin{proof}
		Take $\mathcal{S}\in X_{\calT}$. There is $T\in\mathcal{S}$ such that $0\in\overline{T}$.
		There exist $P\in\mathcal{A}$ and $x\in\Rd$ such that $T=P+x$. Set $\calP=(\mathcal{S}-x)\cap B(0,R)$, then $\calP\in\Pi_{\calT,R,\calA}$.
		Since $-x\in U$, we have $\mathcal{S}\in C_{\calT}(U,\calP)$.
	\end{proof}
	
	\begin{defi}\label{def_main_thm2}
		For a $\mathcal{T}$-legal patch $\mathcal{P}$, $R>0$ and $S\subset\mathbb{R}^d$, set
		\begin{align*}
			\Pi_{\mathcal{T}, R,\calA}(\mathcal{P},S)&=\{\mathcal{P}'\in\Pi_{\mathcal{T},R,\calA}\mid \text{$\mathcal{P}\cup(\mathcal{P}'+x)$ is \textit{not} $\mathcal{T}$-legal
			for any $x\in S$}\}.
		\end{align*}
	\end{defi}

	\begin{thm}\label{main_thm2}
		 Suppose $0$ is a limit point of the set of topological eigenvalues for 
		$\mathbb{R}^d\curvearrowright X_{\mathcal{T}}$. Then for any bounded neighborhood $U$ of $0$ in $\mathbb{R}^d$ which is
		large enough in the sense that $U\supset -\overline{P}$ for all proto-tile $P$, there are $R>0$ and $a\neq 0$ such that
		$\Pi_{\mathcal{T},R,\calA}(\mathcal{P}, x+U+\Ker\chi_a)\neq\emptyset$ for any $\mathcal{P}\in\Pi_{\mathcal{T},R,\calA}$ and $x\in\mathbb{R}^d$.
	\end{thm}

	\begin{proof}
		Let $U$ be a bounded neighborhood of 0 which is large enough. Take $\e\in (0,\frac{1}{4})$.
		For $U$ and $\e$, if $a$ is a topological eigenvalue which is small enough, we can apply Lemma \ref{Lem1_main_thm} to
		$U,a,\e$ and obtain $R>0$ such that $\diam(f(C_{\calT}(U,\calP)))\leq\e$ for any $\calP\in\Pi_{\calT,R,\calA}$.

		Take $\mathcal{P}\in\Pi_{\mathcal{T},R,\calA}$ and $x\in\mathbb{R}^d$, then since $a\neq 0$ and $f$ is surjective,
		$\{f(C_{\calT}(U,\mathcal{P}'))\mid \mathcal{P}'\in\Pi_{\mathcal{T},R,\calA}\}$ covers $\mathbb{T}$.
		Moreover $\diam f(C_{\calT}(U,\mathcal{P}'))<\e$ for each $\mathcal{P}'\in\Pi_{\mathcal{T},R,\calA}$ and so there is 
		$\mathcal{P}'\in\Pi_{\mathcal{T},R,\calA}$ such that
		\begin{align*}
			\chi_a(x)f(C_{\calT}(U,\mathcal{P}))\cap f(C_{\calT}(U,\mathcal{P}'))=\emptyset.
		\end{align*}
		By Lemma \ref{Lem2_main_thm}, $\mathcal{P}\cup(\mathcal{P}'+x+y+z)$ is not $\mathcal{T}$-legal for any $y\in U$ and $z\in\Ker\chi_a$.
		This means $\mathcal{P}'\in\Pi_{\mathcal{T},R,\calA}(\mathcal{P}, x+U+\Ker\chi_a)$.
	\end{proof}
	
	\begin{rem}
		This theorem says the following.  For any large $U$ there are $R$ and $a$ such that we can partially answer the following question:
		given a patch $\calP\in\Pi_{\mathcal{T},R,\calA}$ and a set $S\subset\Rd$, 
		is there a patch in $\Pi_{\mathcal{T},R,\calA}$ 
		such that we can rule out the possibility of its appearance in the area $S$ relative to copies of $\calP$?
		Proposition \ref{prop2_sequence_weak_mixing} which we will show later
		 says that, for certain sequences, 
		 given a word $W$ and large $S$, an appearance of any word is possible.
		 On the contrary, Theorem \ref{main_thm2} states that, for a repetitive FLC tiling of finite tile type which has sufficiently many eigenvalues,
		 given $\calP$ and $x\in\Rd$,
		  for some $\calP'\in\Pi_{\calT,R,\calA}$ an appearance of $\calP'$ in $S=x+U+\Ker\chi_a$ can be ruled out.
	\end{rem}
	
	\begin{rem}
		The cardinality of $\Pi_{\mathcal{T},R,\calA}(\mathcal{P}, x+U+\Ker\chi_a)$ is unknown. This set is the set of patches which cannot be observed
		in the area $x+U+\Ker\chi_a$ relative to a translate of $\mathcal{P}$ in $\calT$. 
		The larger the number 
		$\card\Pi_{\mathcal{T},R,\calA}(\mathcal{P}, x+U+\Ker\chi_a)/\card\Pi_{\mathcal{T},R,\calA}$ is, the more information we have by observing $\calP$.
		It is interesting to research a lower bound of this number for each example of tiling.
	\end{rem}

\section{Tilings from Substitution Rules with Pisot-family assumption}\label{section_substitution_Pisot}

		In Subsection \ref{subsection_Necessary_results_for_self-affine}
		 we introduce substitution rules, which generate interesting examples of tilings called self-affine tilings.
		Next, in Subsection \ref{subsection_Pseudo-self-affine}, we introduce pseudo-substitutions and pseudo-self-affine tilings.
		Pseudo-substitution is also called substitution with amalgamation.
	For tilings constructed from some of these rules, more can be said than Theorem \ref{main_thm1}. 
	We prove the main theorems for certain pseudo-self-affine tilings in Subsection \ref{subsection_main_result_for_pseudo-self-affine}.

	\subsection{Necessary Results for Self-affine Tilings}\label{subsection_Necessary_results_for_self-affine}
		
	\begin{defi}\label{def-substitution}
		A substitution rule of $\mathbb{R}^d$ is a triple $(\mathcal{A}, \phi, \omega)$ where,
		\begin{itemize}
			\item $\mathcal{A}$ is a finite set of tiles in $\mathbb{R}^d$ that contain the origin,
			\item $\phi$ is an expansive linear map, that is, a linear map $\phi\colon\mathbb{R}^d\rightarrow\mathbb{R}^d$ the eigenvalues of which are
			all outside the closed unit disk, and
			\item $\omega$ is a map from $\mathcal{A}$ to 
			\begin{align*}
				\{\mathcal{P}\mid \text{$\mathcal{P}$ is a patch and any $T\in\mathcal{P}$ is a translate of a member of $\mathcal{A}$}\}
			\end{align*} 
			such that
			\begin{align*}
				\supp \omega(P)=\phi \overline{P}
			\end{align*}
			for each $P\in\mathcal{A}$.
		\end{itemize}
	\end{defi}
	
	Elements of $\mathcal{A}$ are called proto-tiles of the substitution. 
	The map $\omega$, called substitution map, 
	is a map obtained by first expanding each proto-tile and then 
	decomposing it to obtain a patch consisting of translates of proto-tiles (see Example \ref{ex_penrose}).
	
	Definition \ref{def-substitution} is for substitution with the group $\mathbb{R}^d$.
	One can also consider a substitution rule for a closed subgroup of $\mathbb{R}^d\rtimes \Od$ bigger than $\mathbb{R}^d$, 
	for example Radin's pinwheel tiling
	(\cite{MR1283873}). We do not deal with such substitutions in this article because we use a property on distribution of eigencharacters
	for our main theorem. If the group for a given substitution is $\mathbb{R}^d$, Theorem \ref{Thm-Solomyak} and Theorem \ref{thm_lee_solomyak}
	are applicable. However the corresponding results for groups bigger than $\mathbb{R}^d$ are not known.
	Hence here we do not deal with substitution rules with groups bigger than $\Rd$.
	
	One can also consider rules that are similar to substitution rules given above, where enlarged tiles are not strictly decomposed, but replaced by
	a patch that does not necessarily cover or is contained in the enlarged tiles (\cite{MR2183221},  \cite{MR1755727},  \cite{MR1854103}).
	In this article we call such rules pseudo-substitution rules, and in the literature it is also called substitution with amalgamation.
	The substitution with kites and darts which generates Penrose tilings is such an example.
	The tilings obtained from such rules are called pseudo-self-affine tilings.
	We deal with pseudo-self-affine tilings in the next subsection.

	We extend the substitution map $\omega$ to translates of proto-tiles by 
	\begin{align}
		\omega(P+x)=\omega(P)+\phi (x)\label{eq_ext_omega}
	\end{align}
	for each $P\in\mathcal{A}$ and $x\in\mathbb{R}^d$.
	For any patch $\mathcal{P}$ consisting of translates of proto-tiles, set $\omega(\mathcal{P})$ by
	\begin{align*}
		\omega(\calP)=\bigcup_{T\in\calP}\omega(T).
	\end{align*}
	The above extension (\ref{eq_ext_omega}) is justified by the fact that $\omega(\calP)$ is again a patch.
	
	\begin{defi}
		Let $(\mathcal{A}, \phi, \omega)$ be a substitution rule. 
		Set
		\begin{align*}
			X_{\omega}=\{\text{$\mathcal{T}$: tiling} \mid \text{for any finite $\mathcal{P}\subset\mathcal{T}$ there are $n>0, P\in\mathcal{A}$ and 
			$x\in\mathbb{R}^d$}\\
			\text{ such that $\mathcal{P}+x\subset\omega^n(P)$}\}.
		\end{align*}
		$\mathbb{R}^d$ acts on this space and the topological dynamical system $(X_{\omega},\mathbb{R}^d)$ is called the corresponding tiling dynamical 
		system for the substitution.
	\end{defi}

	\begin{defi}
		Let $(\mathcal{A}, \phi, \omega)$ be a substitution rule.
		A tiling $\mathcal{T}$ consisting of proto-tiles in $\mathcal{A}$ is called a fixed point of the substitution $(\mathcal{A}, \phi, \omega)$ if 
		$\omega (\mathcal{T})=\mathcal{T}$.
		A repetitive tiling $\calT$ of FLC is called a self-affine tiling if there is a substitution for which $\calT$ is a fixed point.
	\end{defi}
	
	\begin{rem}
		A substitution rule may not have any fixed point, but there are $n>0$ and a tiling $\mathcal{T}$ such that
		$\omega^n(\mathcal{T})=\mathcal{T}$ (Lemma \ref{lem_existence_fixed_point}). Since in many cases we may replace the original substitution 
		$(\mathcal{A}, \phi, \omega)$ with $(\mathcal{A}, \phi^n, \omega^n)$, it suffices to develop a theory for substitutions which 
		admit  fixed points.
	\end{rem}
	

	
	\begin{defi}
		A substitution rule $(\mathcal{A}, \phi, \omega)$ is primitive if there is $n\in\Zpo$ such that, for each $P,P'\in\mathcal{A}$, there is
		 $x\in\mathbb{R}^d$ for which
		$P'+x\in\omega^n(P)$.
	\end{defi}
	
	\begin{lem}\label{lem_existence_fixed_point}
		If a substitution rule $(\mathcal{A}, \phi, \omega)$ 
		is primitive, then there is $n>0$ such that $(\mathcal{A}, \phi^n, \omega^n)$ admits a repetitive fixed point.  
	\end{lem}
	\begin{proof}
	Theorem 5.8, Theorem 5.10 and the following paragraph of \cite{Ro}.
	\end{proof}

	We are interested in substitution rules because they give interesting examples of tilings. Then it is important to know when its repetitive fixed points
	 have FLC.

	\begin{defi}
		A substitution rule $(\mathcal{A},\phi,\omega)$ is of FLC if for any $R>0$, the set
		\begin{align*}
			\{\omega^n(P)\cap B(x,R)\mid P\in\mathcal{A}, n\in\mathbb{Z}_{>0}, x\in\mathbb{R}^d\}/{\simt}
		\end{align*}
		is finite.
	\end{defi}
	
	\begin{lem}
		A substitution rule which has a repetitive fixed point has FLC if and only if any of its repetitive fixed point has FLC.
	\end{lem}

	Here we mention an important consequence of primitivity.
	\begin{thm}[\cite{Solomyak_dynamics}, \cite{MR1976605}]
		If a substitution rule $(\mathcal{A}, \phi, \omega)$ is primitive and has FLC, then the corresponding dynamical system $(X_{\omega}, \mathbb{R}^d)$
		is uniquely ergodic, that is, it admits a unique invariant Borel probability measure.
	\end{thm}
	
	Next we mention the recognizability for self-affine tilings.
	
	\begin{thm}[\cite{MR1637896}]\label{thm_recognizability}
		Let $(\calA,\phi,\omega)$ be a primitive substitution rule of FLC.
		Then $\omega\colon X_{\omega}\rightarrow X_{\omega}$ is injective if and only if there is $\calT\in X_{\omega}$ which is non-periodic.
	\end{thm}

	\begin{ex}[Figure\ref{Penrose_substitution}]\label{ex_penrose}
		Set $\tau=\frac{1+\sqrt{5}}{2}$. Take the interior of the triangle which has side-length 1,1, and $\tau$, and remove one point anywhere from the left.
		Also take the interior of the triangle of the side-length $\tau, \tau$ and $1$, and remove one point from the right. The proto-tiles of this substitution are
		the copies of these two punctured triangles by $2n\pi/10$-rotations and flip, where $n=0,1,\ldots ,9$.
		
		The expansion map is $\tau I$, where $I$ is the identity map. 
		The map $\omega$ is depicted in Figure \ref{Penrose_substitution}. The image of the other proto-tiles
		by $\omega$ is defined accordingly, so that $\omega$ and rotation, $\omega$ and flip will commute.

		Self-affine (in this case self-similar) tilings for this substitution are called Robinson triangle tilings.
 		\begin{figure}[h]
		\begin{center}
		\includegraphics[width=7cm]{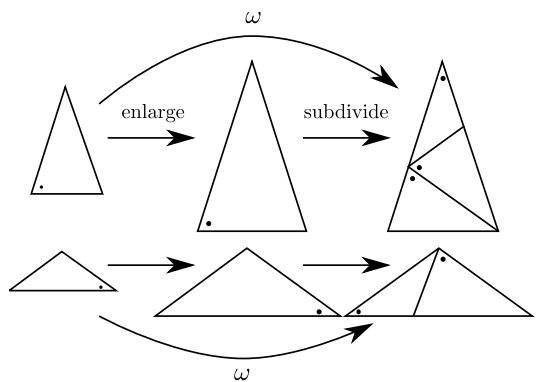}
		\caption{Example of substitution}
		\label{Penrose_substitution}
		\end{center}
		\end{figure}
	\end{ex}
Existence of non-trivial eigenfunctions is assured for tiling dynamical systems coming from certain class of substitution.
	Recall that if $\mathcal{T}$ is a tiling, a vector $z\in\mathbb{R}^d$ is called a return vector for $\mathcal{T}$ if there is 
	$T\in\mathcal{T}$ such that $T+z\in\mathcal{T}$.

	\begin{thm}[\cite{So}, Theorem 3.13]\label{Thm-Solomyak}
		Let $(\mathcal{A}, \phi, \omega)$ be a primitive tiling substitution of FLC.
		Assume there is a repetitive fixed point $\mathcal{T}$. 
		Then for $a\in\mathbb{R}^d$, the following are equivalent:
		\begin{enumerate}
			\item $a$ is a topological eigenvalue for the topological dynamical system $(X_{\omega}, \mathbb{R}^d)$;
			\item $a$ is a measurable eigenvalue for the measure-preserving system $(X_{\omega},\mathbb{R}^d, \mu)$,
				where $\mu$ is the unique invariant probability measure;
			\item $a$ satisfies the following two conditions:
				\begin{enumerate}
					\item For any return vector $z$ for $\mathcal{T}$,  we have
						\begin{align}
							\lim_{n\rightarrow\infty}e^{2\pi i\langle\phi^n(z), a\rangle}=1, \label{1st_condition_for_being_eigenvalue}
						\end{align}
						and
					\item if $z\in\mathbb{R}^d$ and $\mathcal{T}+z=\mathcal{T}$, then
						\begin{align*}
							e^{2\pi i\langle z, a\rangle}=1.
						\end{align*}
				\end{enumerate}
		\end{enumerate}
	\end{thm}
	
	
	\begin{defi}\label{def_epsilon_close}
		Let $F$ and $F'$ be finite subsets of $\Rd$ and $\e$ be a positive real number.
		We say $F$ and $F'$ are $\e$-close if
		\begin{itemize}
			\item for any $x\in F$ there exists a unique $y\in F'$ such that $\|x-y\|<\e$, and
			\item for any $y\in F'$ there exists a unique $x\in F$ such that $\|x-y\|<\e$.
		\end{itemize}
	\end{defi}
	
	\begin{defi}
		Let $S$ be a finite set of algebraic integers. We say $S$ forms a Pisot family if, for any $\lambda\in S$ and 
		its algebraic conjugate $\lambda'$ 
		such that $\lambda'\notin S$, we have $|\lambda'|<1$.
	\end{defi}
	For a linear map $\phi\colon\Rd\rightarrow\Rd$, its adjoint is denoted by $\phi^*$ and its spectrum is denoted by $\spec(\phi)$.

	\begin{thm}[\cite{MR2851885}, Theorem 2.8]\label{thm_lee_solomyak}
		Let $(\mathcal{A}, \phi, \omega)$ be a primitive substitution rule of FLC. 
		Assume $\phi$ is diagonalizable over $\mathbb{C}$ and all the eigenvalues are algebraic conjugates of the same multiplicity.
		Consider the following three conditions:
		\begin{enumerate}
			\item The set $\spec (\phi)$ is a Pisot family.
			\item The set of (topological and measurable) eigenvalues for $(X_{\omega},\Rd)$ is relatively dense.
			\item For any finite set $F\subset\mathbb{R}^d$ and $\e>0$ there is a finite $F'\subset\Rd\setminus\{0\}$ and $n\in\Zpo$ such that
				$F$ and $F'$ are $\e$-close and
			 the set of (topological and measurable) eigenvalues for $(X_{\omega},\Rd)$  includes
				\begin{align*}
					\sum_{x\in F'}\mathbb{Z}[(\phi^*)^{-n}]x.
				\end{align*}
		\end{enumerate}
		Then (1) and (2) are equivalent and (3) implies (1) and (2).
		If moreover $\omega\colon X_{\omega}\rightarrow X_{\omega}$ is injective, (1) and (2) imply (3).
		\end{thm}
	\begin{proof}
		The equivalence of (1) and (2) is proved in \cite[Theorem 2.8]{MR2851885}.
		
		Assume (3). Take a basis $F$ of $\Rd$, then there is $\e>0$ such that if $F'\subset\Rd$ and $F$ and $F'$ are $\e$-close, then
		$F'$ is a basis. For such $F$ and $\e>0$, there are $F'$ and $n$ that satisfy the conditions in (3).
		Then $\sum_{x\in F'}\mathbb{Z}[(\phi^*)^{-n}]x$ consists of eigenvalues and is relatively dense.

		Assume (1), (2) and that $\omega\colon X_{\omega}\rightarrow X_{\omega}$ is injective. 
		There is $n>0$ such that the substitution rule $(\calA,\phi^n,\omega^n)$ admits a repetitive fixed point $\calT$.
		Note that $\calT$ is non-periodic.
		By (2) we can take a basis $\{b_1,b_2,\ldots, b_d\}$ from the set of eigenvalues.
		For any return vector $z$ for $\calT$ 
		and $x\in\{b_1,b_2,\ldots, b_d\}$, the equation \eqref{1st_condition_for_being_eigenvalue} in Theorem \ref{Thm-Solomyak}
		holds. The equation \eqref{1st_condition_for_being_eigenvalue} in Theorem \ref{Thm-Solomyak} also holds for a return vector $z$ and
		$x\in\{(\phi^*)^{-kn}b_i\mid k\in\Zpo, i=1,2,\ldots, d\}$.
		For any $F$ and $\e>0$ take $\e'<\e$ which is small enough so that $B(x,\e')\cap B(y,\e')=\emptyset$ for $x,y\in F$ with $x\neq y$.
		 We can take $k\in\Zpo$ such that
		$\sum_{i=1}^d\|(\phi^*)^{-kn}b_i\|<\e'$.
		For each $x\in F$ we may take $y_x\in\spa_{\mathbb{Z}}\{(\phi^*)^{-kn}b_i\mid i=1,2,\ldots, d\}\setminus\{0\}$ such that 
		$\|x-y_x\|<\e'$. Then $F'=\{y_x\mid x\in F\}$ and $F$ are $\e$-close.
		Moreover elements of the form $(\phi^*)^{-ln}y$, where $y\in F'$ and $l\geq 0$, are eigenvalues and so are their sums.
		\end{proof}

	\subsection{Pseudo-Self-Affine Tilings}\label{subsection_Pseudo-self-affine}
	In order to introduce pseudo-self-affine tilings in this subsection, we briefly recall local derivability between two tilings.
		
		Recall the definition of our notation $\calT\sqcap S$ (Definition \ref{def_cap_sqcap}).
		\begin{defi}[\cite{MR1132337}]
			Let $\calT_1$ and $\calT_2$ be tilings of $\mathbb{R}^d$. 
			We say $\calT_2$ is locally derivable from $\calT_1$ and write $\calT_1\LD\calT_2$ if there is $R>0$ such that
			\begin{align*}
				\text{$x,y\in\mathbb{R}^d$ and $(\calT_1-x)\sqcap B(0,R)=(\calT_1-y)\sqcap B(0,R)$}\\
				\Rightarrow (\calT_2-x)\sqcap\{0\}=(\calT_2-y)\sqcap\{0\}.
			\end{align*}
			If $\calT_1\LD\calT_2$ and $\calT_2\LD\calT_1$, we say $\calT_1$ and $\calT_2$ are mutually locally derivable (MLD) and write
			$\calT_1\MLD\calT_2$.
		\end{defi}
		
		\begin{lem}
			Let $\calT_1$ and $\calT_2$ be tilings of finite tile type in $\mathbb{R}^d$ and suppose $\calT_1\LD\calT_2$.
			Then there is a factor map $X_{\calT_1}\rightarrow X_{\calT_2}$ that sends $\calT_1$ to $\calT_2$.
		\end{lem}
		
		\begin{rem}
			If $\calT_1\MLD\calT_2$, then the two topological dynamical systems $(X_{\calT_1},\mathbb{R}^d)$ and $(X_{\calT_2},\mathbb{R}^d)$
			are topologically conjugate.
			Note also that if $\calT$ is non-periodic and $\calT\MLD\calT'$, then $\calT'$ is also non-periodic.
		\end{rem}
		
		Next we introduce pseudo-substitution (substitution with amalgamation).
		
		\begin{defi}
			A triple $(\calA,\phi,\omega)$ where
			\begin{itemize}
				\item $\calA$ is a finite set of tiles in $\Rd$,
				\item $\phi\colon\Rd\rightarrow\Rd$ is an expansive linear map, and
				\item $\omega\colon\calA\rightarrow\{\text{$\calP$: finite patch}\mid \text{any $T\in\calP$ is a translate of an element of $\calA$}\}$
			\end{itemize}
			is called a pseudo-substitution rule (of $\Rd$).
			
			For a pseudo-substitution rule $(\calA,\phi,\omega)$, $P\in\calA$ and $x\in\Rd$, we set
			\begin{align*}
				\omega(P+x)=\omega(P)+\phi(x).
			\end{align*}
			
			For a patch $\calP$ consisting of translates of elements in $\calA$, we set
			\begin{align*}
				\omega(\calP)=\bigcup_{T\in\calP}\omega(T).
			\end{align*}
		\end{defi}
		
		\begin{defi}\label{def_several_conditions_for_pseudo_substi}
			Let $(\calA,\phi,\omega)$ be a pseudo-substitution rule.
			\begin{itemize}
				\item $(\calA,\phi,\omega)$ is said to be iteratable if for any $P\in\calA$ and $n\in\Zpo$, $\omega^n(P)$ is a patch.
				\item $(\calA,\phi,\omega)$ is said to be primitive if there exists $K\in\Zpo$ such that for any $P$ and $P'\in\calA$ there is $x\in\Rd$ with
					$P'+x\in\omega^K(P)$.
				\item $(\calA,\phi,\omega)$ is said to be expanding if there exist a sequence $(r_m)$ of positive numbers, 
				$P\in\calA, x\in\Rd$ and $n\in\Zpo$ such that
					\begin{itemize}
						\item $P+x\in\omega^n(P+x)$,
						\item $(P+x)+B(0,r_m)\subset\supp\omega^{nm}(P+x)$ for any $m\in\Zpo$, and
						\item $\lim_mr_m=\infty$.
					\end{itemize}
				\item $(\calA,\phi,\omega)$ is said to have FLC if
					\begin{align*}
						\{\omega^n(P)\cap B(x,R)\mid P\in\calA, n\in\Zpo, x\in\Rd\}/{\simt}
					\end{align*}
					is finite for any $R>0$.
			\end{itemize}
		\end{defi}
		
		\begin{rem}
			Substitution rules are pseudo-substitution rules that are iteratable and expanding.
			If a substitution rule is FLC (resp.\ primitive), then it is FLC  (resp.\ primitive) as a pseudo-substitution rule.
		\end{rem}
		
		\begin{rem}\label{rem_bounded_pseudo-substi}
			Since each $\omega(P)$ is finite, there is $L>0$ such that
			\begin{align*}
				\supp\omega(P)\subset\phi(\overline{P})+B(0,L)
			\end{align*}
			for any $P\in\calA$. For any $n\in\Zpo$, we have
			\begin{align*}
				\supp\omega^n(P)\subset(\phi(\overline{P})+B(0,\sum_{k=0}^{n-1}\|\phi\|^kL))
			\end{align*}
			where $\|\phi\|$ is the operator norm of $\phi$.
		\end{rem}

		\begin{prop}
			Let $(\calA,\phi,\omega)$ be an iteratable pseudo-substitution rule.
			Assume $(\calA,\phi,\omega)$ is expanding and so there exist $(r_m),P\in\calA, x\in\Rd$ and $n\in\Zpo$ that satisfy the condition 
			in Definition \ref{def_several_conditions_for_pseudo_substi}. Then
			\begin{align*}
				\calT=\bigcup_{m>0}\omega^{nm}(P+x)
			\end{align*}
			is a tiling such that $\omega^n(\calT)=\calT$.
			
			If $(\calA,\phi,\omega)$ is primitive then $\calT$ is repetitive, and if $(\calA,\phi,\omega)$ has FLC then $\calT$ has FLC.
		\end{prop}
		\begin{proof}
			It is straightforward to show $\calT$ is a tiling with $\omega^n(\calT)=\calT$.
			
			Suppose the substitution rule is primitive. 
			There is $K\in\Zpo$ as in the condition in Definition \ref{def_several_conditions_for_pseudo_substi}.
			Also there is $L>0$ such that $\supp\omega(P)\subset\phi(\overline{P})+B(0,L)$ for each $P\in\calA$ (see Remark \ref{rem_bounded_pseudo-substi}).
			To prove that $\calT$ is repetitive take a finite patch $\calP\subset\calT$.
			There is $m>0$ such that $\calP\subset\omega^{nm}(P+x)$ and $nm\geq K$. 
			Take $R>\max_{P'\in\calA}\diam(\phi^{2nm}(P')+B(0,\sum_{k=0}^{2nm-1}\|\phi\|^kL))$.
			If $y\in\Rd$, there is $T\in\calT$ such that $\phi^{-2nm}y\in\overline{T}$.
			By primitivity there is $z\in\Rd$ such that $P+z\in\omega^{nm}(T)$.
			Then $\calP+\phi^{nm}(z-x)\subset\calT\cap B(y,R)$.
			This shows that
			$\calT$ is repetitive.
			
			The statement with regard to FLC is easy to prove by using Lemma \ref{lem_family_patch_finite}.
		\end{proof}
		
	For a patch $\calP$ of $\mathbb{R}^d$ and an expansive linear map $\phi\colon\mathbb{R}^d\rightarrow\mathbb{R}^d$
	set $\phi(\calP)=\{\phi(T)\mid T\in\calP\}$.
	
	\begin{prop}
		Let $(\calA,\phi,\omega)$ be a pseudo-substitution rule, $\calT$ a tiling consisting of translates of elements in $\calA$ 
		and $n\in\Zpo$. Suppose $\omega^n(\calT)=\calT$.
		Then $\phi^n(\calT)\LD\calT$.
	\end{prop}
	\begin{proof}
		There is $L>0$ such that $\supp\omega(P)\subset\phi(\overline{P})+B(0,L)$ for each $P\in\calA$ (see Remark \ref{rem_bounded_pseudo-substi}).
		Set $R=\max_{P\in\calA}\diam(\phi^n(\overline{P})+B(0,\sum_{k=0}^{n-1}\|\phi\|^kL))$. If $x\in\Rd$, $T\in\calT$ and $x\in\overline{T}$,then 
		$T\in\omega^n(\calT\cap\phi^{-n}B(x,R))$. 
		 To prove local derivability suppose $x,y\in\Rd$ and $(\phi^n(\calT)-x)\sqcap B(0,R)=(\phi^n(\calT)-y)\sqcap B(0,R)$. Then
		$\calT\cap\phi^{-n}(B(x,R))-\phi^{-n}(x)=\calT\cap\phi^{-n}(B(y,R))-\phi^{-n}(y)$. This implies that
		\begin{align*}
			\calT\sqcap\{x\}&=\omega^n(\calT\cap\phi^{-n}(B(x,R)))\sqcap\{x\}\\
						&=(\omega^n(\calT\cap\phi^{-n}(B(y,R)))\sqcap\{y\})+x-y\\
						&=\calT\sqcap\{y\}+x-y.
		\end{align*}
	\end{proof}

	\begin{defi}[\cite{MR2183221}]
		A repetitive FLC tiling $\calT$ is called a pseudo-self-affine tiling if $\phi(\calT)\LD\calT$
		for some expansive linear map $\phi\colon\mathbb{R}^d\rightarrow\mathbb{R}^d$.
	\end{defi}

	\begin{rem}
		Self-affine tilings are pseudo-self-affine tilings.
	\end{rem}
	
	\begin{thm}[\cite{MR2183221}]\label{thm_pseudo_affine}
		Let $\calT$ be a pseudo-self-affine tilings with an expansion map $\phi$.
		Then for any  $k\in\Zpo$ sufficiently large, there exists a tiling $\calT'$ which is self-affine with expansion $\phi^k$
		such that $\calT$ is MLD with $\calT'$.
	\end{thm}

	\subsection{Main Results for Pseudo-Self-Similar Tilings}\label{subsection_main_result_for_pseudo-self-affine}
	Since self-affine tilings are pseudo-self-affine tilings, we can replace the word ``pseudo-self-affine'' with ``self-affine'' in the results in
	this subsection.
	\begin{thm}\label{thm1_substitution}
		Let $\calT$ be a non-periodic pseudo-self-affine tiling for an expansion map $\phi$.
		Assume $\phi$ is diagonalizable over $\mathbb{C}$, all the eigenvalues are algebraic conjugates of the same multiplicity, and $\spec(\phi)$ forms 
		a Pisot family.
		Then for any $R_0>0$, $\e>0$ and a finite subset $F$ of $B(0,\frac{1}{8R_0})$, there exist $R>0$, a finite subset $F'\subset\Rd\setminus\{0\}$,
		and $x_a:\bigcup_{L>R}\Pi_{\calT,L,\calA}\rightarrow\mathbb{R}^d$  for each $a\in F'$
		such that the following two conditions hold:
		\begin{enumerate}
			\item $F$ and $F'$ are $\e$-close.
			\item the patch	
				\begin{align*}
					(\calP_1+x_a(\calP_1))\cup (\calP_2+x_a(\calP_2)+\frac{1}{2\| a\|^2}a+y+z)
				\end{align*}
				is \emph{not} $\calT$-legal for any $a\in F'$, 
				$\mathcal{P}_1,\mathcal{P}_2\in\bigcup_{L>R}\Pi_{\mathcal{T},L,\calA}$, $y\in\Ker\chi_{a}$ and $z\in B(0,R_0)$.
		\end{enumerate}
	\end{thm}
	\begin{proof}
	By Theorem \ref{thm_pseudo_affine}, there are $k>0$ and a self-affine tiling $\calT'$ with an expansion map $\phi^k$ such that $\calT\MLD\calT'$.
	Note that $\calT'$ is non-periodic and by Theorem \ref{thm_recognizability}, Theorem \ref{thm_lee_solomyak} is applicable.
	Take $\eta\in (\max_{a\in F}\|a\|,\frac{1}{8R_0})$.
	Since $(X_{\calT},\Rd)$ and $(X_{\calT'},\Rd)$ are topologically conjugate,
	by Theorem \ref{thm_lee_solomyak} we can take a finite $F'\subset\Rd\setminus\{0\}$ such that the following two conditions hold:
	\begin{itemize}
		\item $F$ and $F'$ are $\min\{\e, \frac{1}{8R_0}-\eta\}$-close.
		\item any $a\in F'$ is a topological eigenvalue.
	\end{itemize}
	By Proposition \ref{prop_main_thm}, for any $a\in F'$ there exist $R_a>0$ and a map $x_a'\colon\bigcup_{L>R}\Pi_{\calT,L,\calA}\rightarrow\Rd$
	such that the patch
	\begin{align*}
		(\calP_1+x_a'(\calP_1))\cup (\calP_2+x_a'(\calP_2)+\frac{1}{2\|a\|^2}a+y+z)
	\end{align*}	
	is \emph{not} $\calT$-legal for any $\calP_1,\calP_2\in\bigcup_{L>R_a}\Pi_{\calT,L,\calA}$, $y\in\Ker\chi_a$ and $z\in B(0,R_0)$.
	Set $R=\max_{a\in F'}R_a$ and $x_a=x_a'|_{\bigcup_{L>R}\Pi_{\calT,L,\calA}}$ (the restriction). Then these $R$ and $x_a$ have the desired property.
	\end{proof}

	\begin{cor}\label{cor_substi_main_thm1}
		Let $\calT$ be a non-periodic pseudo-self-affine tiling with an expansion map $\phi$.
		Assume that $\phi$ is diagonalizable over $\mathbb{C}$ and all the eigenvalues are algebraic conjugates of the same multiplicity.
		Assume also that $\spec(\phi)$ forms a Pisot family.
		Then for any $R_0>0$ and $\e>0$, there exist a basis $\mathcal{B}$ of $\Rd$, $R>0$ and $x_a\colon\bigcup_{L>R}\Pi_{\calT,L,\calA}\rightarrow \Rd$
		 for each $a\in\mathcal{B}$ such that the following two conditions hold:
		 \begin{enumerate}
		 	\item The patch
				\begin{align*}
					(\calP_1+x_a(\calP_1))\cup (\calP_2+x_a(\calP_2)+\frac{1}{2\| a\|^2}a+y+z)
				\end{align*}
				is \emph{not} $\calT$-legal for any $a\in \mathcal{B}$, 
				$\mathcal{P}_1,\mathcal{P}_2\in\bigcup_{L>R}\Pi_{\mathcal{T},L,\calA}$, $y\in\Ker\chi_{a}$ and $z\in B(0,R_0)$.
			\item $8R_0<\frac{1}{\|a\|}<(8+\e)R_0$ for each $a\in\mathcal{B}$.
		 \end{enumerate} 
	\end{cor}
	
	\begin{proof}
		Take a basis $\mathcal{B}_0$ of $\Rd$ such that if $a\in\mathcal{B}_0$, then $\|a\|\in(\frac{1}{R_0(8+\e)}, \frac{1}{8R_0})$.
		We can take $\e'>0$ such that
		\begin{itemize}
			\item if $b\in\Rd$, $a\in\mathcal{B}_0$ and $\|a-b\|<\e'$, then we have $\|b\|\in(\frac{1}{R_0(8+\e)}, \frac{1}{8R_0})$, and
			\item if we take $b_a\in\Rd$ such that $\|a-b_a\|<\e'$ for each $a\in\mathcal{B}_0$, then $\{b_a\mid a\in\mathcal{B}_0\}$ is a basis.
		\end{itemize}
		Apply Theorem \ref{thm1_substitution} to $R_0, \e'$ and $\mathcal{B}_0$. We obtain $\mathcal{B}$, $R>0$ and maps $x_a$.
		These satisfy the desired conditions.
	\end{proof}

			\begin{figure}[h]
		\begin{center}
		\includegraphics[width=1.0\columnwidth]{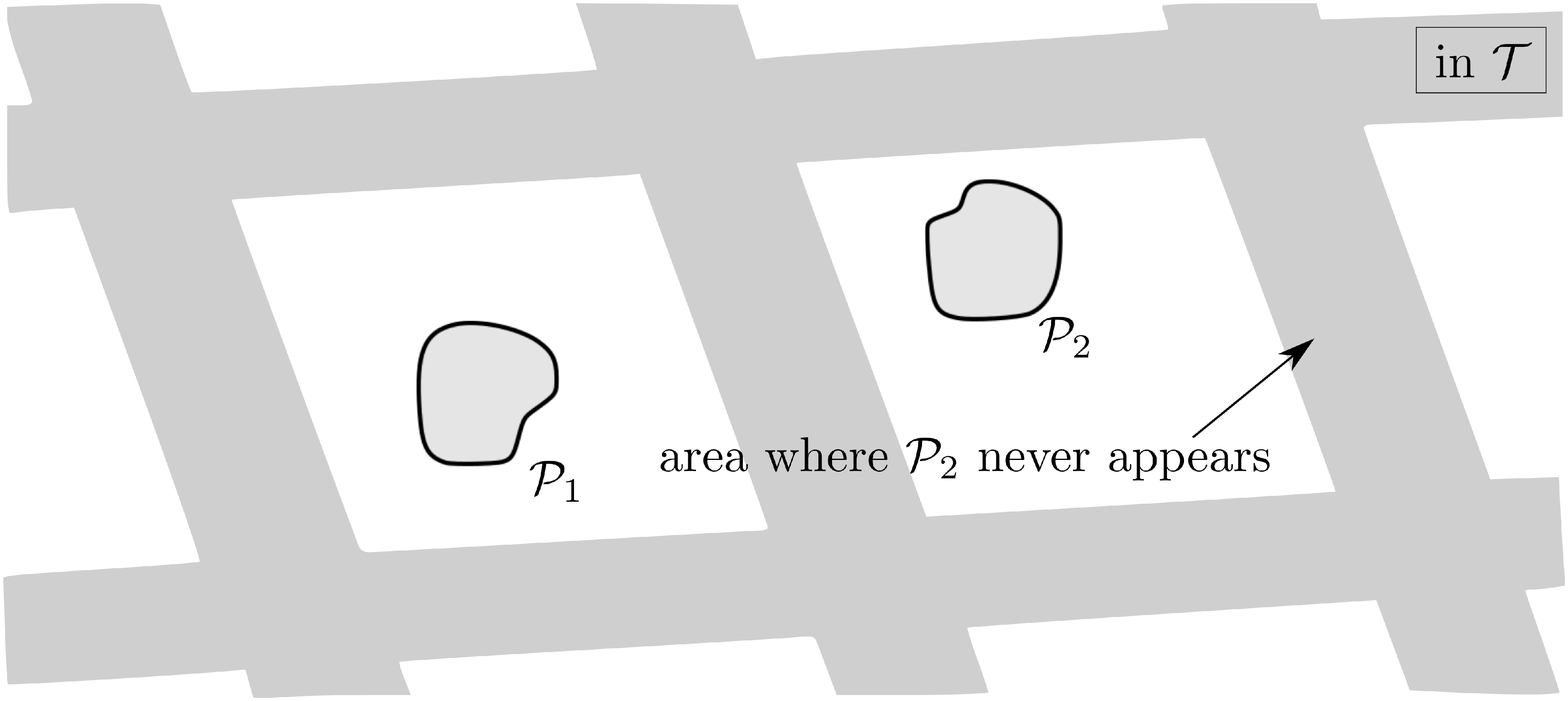}%
		\caption{Corollary \ref{cor_substi_main_thm1}}
		\label{figure_cor_thm1_substitution}
		\end{center}
		\end{figure}
		
	\begin{rem}
		The claim of Corollary \ref{cor_substi_main_thm1} is depicted in Figure \ref{figure_cor_thm1_substitution} in page \pageref{figure_cor_thm1_substitution}.
		It says that, the forbidden area is grid-like, the width of each band is $2R_0$ and the intervals of bands are approximately $8R_0$.
	\end{rem}
	Recall the definition of $\Pi_{\calT,\calA}$ (Definition \ref{def_language}).
	
	\begin{thm}\label{thm2_substitution}
		Let $(\mathcal{A}, \phi, \omega)$ be an iteratable, primitive and expanding pseudo-substitution of FLC
		 which has a repetitive, non-periodic and FLC fixed point $\mathcal{T}$.
		Assume $\phi$ is diagonalizable over $\mathbb{C}$ and all the eigenvalues are algebraic conjugates of the same multiplicity.
		Also assume $\spec(\phi)$ is a Pisot family.
		Then for any $R_0>0$, $\e>0$ and a finite subset $F\subset B(0,\frac{1}{8R_0})$,
		 there are $m\in\mathbb{Z}_{>0}$, a finite subset $F'$ of $\mathbb{R}^d\setminus\{0\}$ and 
		$x_a\colon \Pi_{\calT,\calA}\rightarrow \mathbb{R}^d$ for each $a\in F'$ such that the following two conditions hold:
		\begin{itemize}
			\item For each $a\in F'$, $\calP_1,\calP_2\in\Pi_{\calT,\calA}, y\in\Ker\chi_{a}$ and $z\in B(0,R_0)$,
				\begin{align*}
					 (\calP_1+x_a(\calP_1))\cup(\calP_2+x_a(\calP_2)+\phi^{-m}(\frac{1}{2\|a\|^2}a+y+z))
				\end{align*}
				is \emph{not} $\calT$-legal.
			\item $F$ and $F'$ are $\e$-close.
		\end{itemize}
	\end{thm}
	\begin{proof}
		For any $R_0, \e, F$,  there are $R>0$,  $F'$ and $x'_a$ as in Theorem \ref{thm1_substitution}. 
		Take $m\in \mathbb{Z}_{>0}$ such that for any $T\in\mathcal{A}$, there are $\calP_{T}\in\bigcup_{L>R}\Pi_{\calT,L,\calA}$ and
		$y_T\in\mathbb{R}^d$ such that $\omega^m(\{T\})\supset \calP_T+y_T$.
		Set $x_a(T)=\phi^{-m}(x'_a(\calP_T)-y_T)$.
		
		Take $a\in F'$, $T,T'\in\mathcal{A}$, $y\in\Ker\chi_a$ and $z\in B(0,R_0)$. Suppose a patch
		\begin{align*}
			\mathcal{Q}_{T,T',y,z,a}=(\{T\}+x_a(T))\cup(\{T'\}+x_a(T')+\phi^{-m}(\frac{1}{2\|a\|^2}a+y+z))
		\end{align*}
		is $\calT$-legal. Then
		\begin{align*}
			(\omega^m(\{T\})+\phi^m (x_a(T)))\cup(\omega^m(\{T'\})+\phi^m (x_a(T'))+\frac{1}{2\|a\|^2}a+y+z)
		\end{align*}
		is also $\calT$-legal. It follows that
		\begin{align*}
			(\calP_T+x'_a(\calP_T))\cup(\calP_{T'}+x'_a(\calP_{T'})+\frac{1}{2\|a\|^2}a+y+z)
		\end{align*}
		is also $\calT$-legal, which is a contradiction. Hence $\mathcal{Q}_{T,T',y,z,a}$ is not $\calT$-legal. For $\calP\in\Pi_{\calT,\calA}$ there is 
		$T(\calP)\in\calP\cap\mathcal{A}$. Set $x_a(\calP)=x_a(T(\calP))$. For any $a\in F', 
		\calP_1, \calP_2\in\Pi_{\calT,\calA}, y\in\Ker\chi_a$ and $z\in B(0,R_0)$, the patch
		\begin{align*}
			(\calP_1+x_a(\calP_1))\cup(\calP_2+x_a(\calP_2)+\phi^{-m}(\frac{1}{2\|a\|}a+y+z))
		\end{align*}
		is not $\calT$-legal because it includes $\mathcal{Q}_{T(\calP_1),T(\calP_2),y,z,a}$.
	\end{proof}

	\begin{rem}
		In Theorem \ref{thm2_substitution}, each ``bands'' of the area where $\calP_2$ 
		never appears may be thin. However in this theorem we can see the relative relation
		of any two patches ($\calP_1$ and $\calP_2$) no matter how small they are (see Figure \ref{figure_introduction} in page \pageref{figure_introduction}). 
		Note that in Theorem \ref{thm1_substitution} the patches we deal with must be large enough.
	\end{rem}

	\begin{thm}\label{thm3_substitution}
		Let $\calT$ be a non-periodic self-affine tiling with expansion map $\phi$.
		Assume $\phi$ is diagonalizable over $\mathbb{C}$, all the eigenvalues are algebraic conjugates of the same multiplicity,
		and $\spec(\phi)$ forms a Pisot family.
		Then for any bounded neighborhood $U$ of $0$ in $\mathbb{R}^d$ which is
		large enough in the sense that $U\supset -\overline{P}$ for all proto-tile $P$, there are $R>0$ and $a\neq 0$ such that
		$\Pi_{\mathcal{T},R,\calA}(\mathcal{P}, x+U+\Ker\chi_a)\neq\emptyset$ for any $\mathcal{P}\in\Pi_{\mathcal{T},R,\calA}$ and $x\in\mathbb{R}^d$.
	\end{thm}
	\begin{proof}
	Such tilings satisfy the assumption of Theorem \ref{main_thm2} by  Theorem \ref{thm_lee_solomyak}.
	\end{proof}

	\begin{rem}
		We can apply Theorem \ref{thm1_substitution}, Corollary \ref{cor_substi_main_thm1} and
		Theorem \ref{thm2_substitution} 
		to many examples of pseudo-substitutions and pseudo-self-affine tilings. 
		We do not need to assume that the tilings are MLD with model sets, although if Pisot conjecture is true,
		then all the irreducible Pisot substitutions give rise to tilings which are MLD with model sets.
		Examples  for our results include the one in Example \ref{ex_penrose}, chair substitution, 
		``Penrose kite and dart'' and Ammann-Beenker substitution.
	\end{rem}

\section{Sequences with Weakly Mixing Dynamical Systems}\label{section_sequence_weak_mixing}

	Here we briefly mention that for sequences with weakly mixing dynamical systems, the situation is contrary to the one in the
	previous sections. 
	
	First we briefly recall word substitution. For details see \cite{MR2590264}.
	Note that we define a subshift as a subset of $\mathcal{A}^{\mathbb{Z}}$, not of $\mathcal{A}^{\Znn}$, but this change does not give an immense effect.
	\begin{defi}\label{def_language_for_sequences}
		Let $\mathcal{A}$ be a finite set. For $x\in\mathcal{A}^{\mathbb{Z}}$,
		define its language by 
		\begin{align*}
		\mathcal{L}_x=\{x_nx_{n+1}\cdots x_{n+m}\mid n\in\mathbb{Z}, m\in\Znn\}.
		\end{align*}
		To $\mathcal{A}^{\mathbb{Z}}$, the product topology is given.
		For $x$, its orbit closure is defined by $\overline{\mathcal{O}(x)}=\overline{\{T^n(x)\mid n\in\mathbb{Z}\}}$ where $T$ is the left-shift: $T((x_n))=(y_n)$,
		$x_{n+1}=y_n$.
	\end{defi}
	
	\begin{defi}
		For a finite set $\mathcal{A}$, a map $\zeta\colon\mathcal{A}\rightarrow \mathcal{A}^+=\bigcup_{k>0}\mathcal{A}^k$ is called a word substitution.
		Its language is defined by 
		\begin{align*}
		\mathcal{L}_{\zeta}=\{\text{$W$:word}\mid \text{$W$: a factor of a $\zeta^n(a)$ for $n>0$ and $a\in\mathcal{A}$}\}.
		\end{align*}
		Set $X_{\zeta}=\{x\in\mathcal{A}^{\mathbb{Z}}\mid \mathcal{L}_x\subset\mathcal{L}_{\zeta}\}$ and 
		for a word $W$, define its cylinder set by $[W]=\{x\in X_{\zeta}\mid x_0x_1\cdots x_{|W|-1}=W\}$.
	\end{defi}
	
	\begin{rem}	
	 It is known that if $\zeta$ satisfies the condition called primitivity,
	$(X_{\zeta},T)$ is minimal and  has a unique invariant probability
	measure $\mu$.
	By Dekking-Keane \cite{MR0466485} we see there is a primitive word substitution such that $(X_{\zeta},T,\mu)$ is weakly mixing.
	In \cite{MR0466485} $X_{\zeta}$ was defined as a subset of $\mathcal{A}^{\Znn}$, 
	but the property that the dynamical system is weakly mixing is equivalent between these two definitions.
	\end{rem}

	
	
	Recall a subset $J\subset \Znn$ is said to have density zero if $\limsup_{N\rightarrow\infty}\frac{\card J\cap [0,N]}{N+1}=0$.
	\begin{prop}\label{prop1_sequence_weak_mixing}
		Let $x\in\mathcal{A}^{\mathbb{Z}}$ be a sequence such that the corresponding dynamical system $(\overline{\mathcal{O}(x)},T)$ has an invariant 
		probability measure
		$\mu$ with $\supp\mu=\overline{\mathcal{O}(x)}$ and $(\overline{\mathcal{O}(x)}, T,\mu)$ is weakly mixing.
		Then there is a subset $J(W_1,W_2)\subset\Znn$ of density zero for each $W_1,W_2\in\mathcal{L}_{x}$ 
		such that for any $W_1,W_2\in\mathcal{L}_{x}$ and
		$n\in\Znn\setminus J(W_1,W_2)$, $[W_1]\cap T^{-n}[W_2]\neq\emptyset$.
	\end{prop}
	
	\begin{proof}
		This is just an application of one of characterizations of weakly mixing.
	\end{proof}
	\begin{rem}\label{rem_after_prop1_sequence}
		This proposition says that the relative position of words is ``freer''  than in the situation of Theorem \ref{main_thm1}.
		It says that if we observe a copy of $W_1$ in $x$, we cannot rule out the possibility of appearance of $W_2$ 
		in a wide area (complement of $J(W_1, W_2)$)
		relative to this copy of $W_1$ .
		It can be said that the ``patch'' $W_1\cup(W_2+n)$ is ``legal'' for any $n\notin J(W_1,W_2)$ if the terms are appropriately defined.

		Proposition \ref{prop1_sequence_weak_mixing} is applicable to any element $x\in X_{\zeta}$ where $\zeta$ is a primitive word substitution 
		such that $(X_{\zeta}, T,\mu)$ is weakly mixing, $\mu$ being the unique invariant probability measure.
	\end{rem}
	
	\begin{defi}
		For a word substitution $\zeta$ and $m\in\Zpo$, set $\mathcal{L}_{\zeta,m}=\{W\in\mathcal{L}_{\zeta}\mid |W|=m\}. $
	\end{defi}
	
	\begin{prop}\label{prop2_sequence_weak_mixing}
		Let $x\in\mathcal{A}^{\mathbb{Z}}$ be a sequence such that the corresponding dynamical system $(\overline{\mathcal{O}(x)},T)$ has an invariant 
		probability measure
		$\mu$ with $\supp\mu=\overline{\mathcal{O}(x)}$ and $(\overline{\mathcal{O}(x)}, T,\mu)$ is weakly mixing.
		Let $J(W_1,W_2)$ be the subset of $\Znn$ described in
		Proposition \ref{prop1_sequence_weak_mixing}. Set $J(W)=\bigcup_{W'\in\mathcal{L}_{\zeta,|W|}}J(W,W')$.
		If $m\in\Zpo, W_1,W_2\in\mathcal{L}_{\zeta,m}$ and $n\in\Znn\setminus J(W_1)$, 
		$[W_1]\cap T^{-n}[W_2]\neq\emptyset.$
	\end{prop}
	\begin{proof}
	Clear by Proposition \ref{prop1_sequence_weak_mixing}.
	\end{proof}

	\begin{rem}
		This proposition says that, for certain sequences, the situation on distribution of words is contrary to the one in Theorem \ref{main_thm2}.
		That is, given a word $W_1$ and $n\notin J(W_1)$, if we observe  $W_1$ in the sequence and move our attention by $n$ to the right,
		any word can be observed and no possibility can be ruled out.
		(Note that each $J(W)$ has density zero.)
		In Theorem \ref{main_thm2},  some possibility can be ruled out.
	\end{rem}
			
\textbf{Acknowledgements.} I thank Takeshi Katsura for helpful comments and discussions. I also thank Shigeki Akiyama for interesting comments and discussions.
	Finally I thank referees for comments by which the article was improved.


\bibliographystyle{amsplain}
\bibliography{tiling}

\providecommand{\bysame}{\leavevmode\hbox to3em{\hrulefill}\thinspace}
\providecommand{\MR}{\relax\ifhmode\unskip\space\fi MR }
\providecommand{\MRhref}[2]{%
  \href{http://www.ams.org/mathscinet-getitem?mr=#1}{#2}
}
\providecommand{\href}[2]{#2}
\begin{thebibliography}{10}

\bibitem{AP}
Jared~E. Anderson and Ian~F. Putnam, \emph{Topological invariants for
  substitution tilings and their associated {$C^*$}-algebras}, Ergodic Theory
  Dynam. Systems \textbf{18} (1998), no.~3, 509--537. \MR{1631708
  (2000a:46112)}

\bibitem{MR1132337}
M.~Baake, M.~Schlottmann, and P.~D. Jarvis, \emph{Quasiperiodic tilings with
  tenfold symmetry and equivalence with respect to local derivability}, J.
  Phys. A \textbf{24} (1991), no.~19, 4637--4654. \MR{1132337 (92k:52043)}

\bibitem{MR2106769}
Michael Baake and Daniel Lenz, \emph{Dynamical systems on translation bounded
  measures: pure point dynamical and diffraction spectra}, Ergodic Theory
  Dynam. Systems \textbf{24} (2004), no.~6, 1867--1893. \MR{2106769
  (2005h:37007)}

\bibitem{MR1798986}
Michael Baake and Robert~V. Moody (eds.), \emph{Directions in mathematical
  quasicrystals}, CRM Monograph Series, vol.~13, American Mathematical Society,
  Providence, RI, 2000. \MR{1798986 (2001f:52047)}

\bibitem{BH}
Nicolas Bedaride and Arnaud Hilion, \emph{Geometric realizations of two
  dimensional substitutive tilings, arxiv:1101.3905v4 [math.gt]}.

\bibitem{MR0466485}
F.~M. Dekking and M.~Keane, \emph{Mixing properties of substitutions}, Z.
  Wahrscheinlichkeitstheorie und Verw. Gebiete \textbf{42} (1978), no.~1,
  23--33. \MR{0466485 (57 \#6363)}

\bibitem{grimm2008homometric}
Uwe Grimm and Michael Baake, \emph{Homometric point sets and inverse problems},
  Zeitschrift f{\"u}r Kristallographie International journal for structural,
  physical, and chemical aspects of crystalline materials \textbf{223} (2008),
  no.~11-12, 777--781.

\bibitem{hof1995diffraction}
A~Hof, \emph{On diffraction by aperiodic structures}, Communications in
  mathematical physics \textbf{169} (1995), no.~1, 25--43.

\bibitem{MR1976605}
Jeong-Yup Lee, Robert~V. Moody, and Boris Solomyak, \emph{Consequences of pure
  point diffraction spectra for multiset substitution systems}, Discrete
  Comput. Geom. \textbf{29} (2003), no.~4, 525--560. \MR{1976605 (2005g:37026)}

\bibitem{MR2851885}
Jeong-Yup Lee and Boris Solomyak, \emph{Pisot family self-affine tilings,
  discrete spectrum, and the {M}eyer property}, Discrete Contin. Dyn. Syst.
  \textbf{32} (2012), no.~3, 935--959. \MR{2851885 (2012h:37040)}

\bibitem{lenz2011stationary}
Daniel Lenz and Robert~V Moody, \emph{Stationary processes with pure point
  diffraction}, arXiv preprint arXiv:1111.3617 (2011).

\bibitem{MR1854103}
N.~Priebe and B.~Solomyak, \emph{Characterization of planar pseudo-self-similar
  tilings}, Discrete Comput. Geom. \textbf{26} (2001), no.~3, 289--306.
  \MR{1854103 (2002j:37029)}

\bibitem{MR1755727}
Natalie~M. Priebe, \emph{Towards a characterization of self-similar tilings in
  terms of derived {V}orono\u\i\ tessellations}, Geom. Dedicata \textbf{79}
  (2000), no.~3, 239--265. \MR{1755727 (2001e:37025)}

\bibitem{MR2590264}
Martine Queff{\'e}lec, \emph{Substitution dynamical systems---spectral
  analysis}, second ed., Lecture Notes in Mathematics, vol. 1294,
  Springer-Verlag, Berlin, 2010. \MR{2590264 (2011b:37018)}

\bibitem{MR1283873}
Charles Radin, \emph{The pinwheel tilings of the plane}, Ann. of Math. (2)
  \textbf{139} (1994), no.~3, 661--702. \MR{1283873 (95d:52021)}

\bibitem{Ro}
E.~Arthur Robinson, Jr., \emph{Symbolic dynamics and tilings of {$\Bbb R^d$}},
  Symbolic dynamics and its applications, Proc. Sympos. Appl. Math., vol.~60,
  Amer. Math. Soc., Providence, RI, 2004, pp.~81--119. \MR{2078847
  (2005h:37036)}

\bibitem{shechtman1984metallic}
Dan Shechtman, Ilan Blech, Denis Gratias, and John~W Cahn, \emph{Metallic phase
  with long-range orientational order and no translational symmetry}, Physical
  Review Letters \textbf{53} (1984), no.~20, 1951.

\bibitem{MR1637896}
B.~Solomyak, \emph{Nonperiodicity implies unique composition for self-similar
  translationally finite tilings}, Discrete Comput. Geom. \textbf{20} (1998),
  no.~2, 265--279. \MR{1637896 (99f:52028)}

\bibitem{MR2183221}
\bysame, \emph{Pseudo-self-affine tilings in {$\Bbb R^d$}}, Zap. Nauchn. Sem.
  S.-Peterburg. Otdel. Mat. Inst. Steklov. (POMI) \textbf{326} (2005),
  no.~Teor. Predst. Din. Sist. Komb. i Algoritm. Metody. 13, 198--213,
  282--283. \MR{2183221 (2006k:52053)}

\bibitem{Solomyak_dynamics}
Boris Solomyak, \emph{Dynamics of self-similar tilings}, Ergodic Theory Dynam.
  Systems \textbf{17} (1997), no.~3, 695--738. \MR{1452190 (98f:52030)}

\bibitem{So}
\bysame, \emph{Eigenfunctions for substitution tiling systems}, Probability and
  number theory---{K}anazawa 2005, Adv. Stud. Pure Math., vol.~49, Math. Soc.
  Japan, Tokyo, 2007, pp.~433--454. \MR{2405614 (2010b:37037)}

\bibitem{terauds2013inverse}
Venta Terauds, \emph{The inverse problem of pure point diffraction--examples
  and open questions}, Journal of Statistical Physics \textbf{152} (2013),
  no.~5, 954--968.

\bibitem{terauds2013some}
Venta Terauds and Michael Baake, \emph{Some comments on the inverse problem of
  pure point diffraction}, Aperiodic Crystals, Springer, 2013, pp.~35--41.

\bibitem{Wh}
Michael~F. Whittaker, \emph{{$C^*$}-algebras of tilings with infinite
  rotational symmetry}, J. Operator Theory \textbf{64} (2010), no.~2, 299--319.
  \MR{2718945 (2011m:46127)}

\end{thebibliography}

\end{document}